\documentclass[10pt,twoside]{article}
\usepackage{amsmath, amsfonts, amssymb, mathrsfs}
\usepackage[curve, arrow]{xypic}
\usepackage{mathrsfs}

\newtheorem{theorem}{Theorem}
\newtheorem{proposition}[theorem]{Proposition}
\newtheorem{lemma}[theorem]{Lemma}
\newtheorem{corollary}[theorem]{Corollary}
\newtheorem{definition}{Definition}

\newenvironment{proof}{{ \emph{Proof.} }}{\hfill $\Box$ \\}

\def\YEAR{\year}\newcount\VOL\VOL=\YEAR\advance\VOL by-1995
\def\firstpage{1}\def\lastpage{5}
\def\received{}\def\revised{}
\def\communicated{}

\makeatletter
\def\magnification{\afterassignment\m@g\count@}
\def\m@g{\mag=\count@\hsize6.5truein\vsize8.9truein\dimen\footins8truein}
\makeatother

\font\eightrm=cmr8
\font\caps=cmcsc10                    
\font\Caps=cmcsc10 scaled \magstep1   

\makeatletter
\@specialpagefalse\headheight=8.5pt
\def\DocMath{}
\renewcommand{\@evenhead}{%
    \ifnum\thepage>\lastpage\rlap{\thepage}\hfill%
    \else\rlap{\thepage}\slshape\leftmark\hfill{\caps\SAuthor}\hfill\fi}%
\renewcommand{\@oddhead}{%
    \ifnum\thepage=\firstpage{\DocMath\hfill\llap{\thepage}}%
    \else{\slshape\rightmark}\hfill{\caps\STitle}\hfill\llap{\thepage}\fi}%
\makeatother

\def\TSkip{\bigskip}
\newbox\TheTitle{\obeylines\gdef\GetTitle #1
\ShortTitle  #2
\SubTitle    #3
\Author      #4
\ShortAuthor #5
\EndTitle
{\setbox\TheTitle=\vbox{\baselineskip=20pt\let\par=\cr\obeylines%
\halign{\centerline{\Caps##}\cr\noalign{\medskip}\cr#1\cr}}%
	\copy\TheTitle\TSkip\TSkip%
\def\next{#2}\ifx\next\empty\gdef\STitle{#1}\else\gdef\STitle{#2}\fi%
\def\next{#3}\ifx\next\empty%
    \else\setbox\TheTitle=\vbox{\baselineskip=20pt\let\par=\cr\obeylines%
    \halign{\centerline{\caps##} #3\cr}}\copy\TheTitle\TSkip\TSkip\fi%
\centerline{\caps #4}\TSkip\TSkip%
\def\next{#5}\ifx\next\empty\gdef\SAuthor{#4}\else\gdef\SAuthor{#5}\fi%
\ifx\received\empty\relax
    \else\centerline{\eightrm Received: \received}\fi%
\ifx\revised\empty\TSkip%
    \else\centerline{\eightrm Revised: \revised}\TSkip\fi%
\ifx\communicated\empty\relax
    \else\centerline{\eightrm Communicated by \communicated}\fi\TSkip\TSkip%
\catcode'015=5}}\def\Title{\obeylines\GetTitle}
\def\Abstract{\begingroup\narrower
    \parskip=\medskipamount\parindent=0pt{\caps Abstract. }}
\def\EndAbstract{\par\endgroup\TSkip}

\long\def\MSC#1\EndMSC{\def\arg{#1}\ifx\arg\empty\relax\else
     {\par\narrower\noindent%
     2000 Mathematics Subject Classification: #1\par}\fi}

\long\def\KEY#1\EndKEY{\def\arg{#1}\ifx\arg\empty\relax\else
	{\par\narrower\noindent Keywords and Phrases: #1\par}\fi\TSkip}

\newbox\TheAdd\def\Addresses{\vfill\copy\TheAdd\vfill
    \ifodd\number\lastpage\vfill\eject\phantom{.}\vfill\eject\fi}
{\obeylines\gdef\GetAddress #1
\Address #2
\Address #3
\Address #4
\EndAddress
{\def\xs{4.3truecm}\parindent=0pt
\setbox0=\vtop{{\obeylines\hsize=\xs#1\par}}\def\next{#2}
\ifx\next\empty 
     \setbox\TheAdd=\hbox to\hsize{\hfill\copy0\hfill}
\else\setbox1=\vtop{{\obeylines\hsize=\xs#2\par}}\def\next{#3}
\ifx\next\empty 
     \setbox\TheAdd=\hbox to\hsize{\hfill\copy0\hfill\copy1\hfill}
\else\setbox2=\vtop{{\obeylines\hsize=\xs#3\par}}\def\next{#4}
\ifx\next\empty\ 
     \setbox\TheAdd=\vtop{\hbox to\hsize{\hfill\copy0\hfill\copy1\hfill}
                \vskip20pt\hbox to\hsize{\hfill\copy2\hfill}}
\else\setbox3=\vtop{{\obeylines\hsize=\xs#4\par}}
     \setbox\TheAdd=\vtop{\hbox to\hsize{\hfill\copy0\hfill\copy1\hfill}
	        \vskip20pt\hbox to\hsize{\hfill\copy2\hfill\copy3\hfill}}
\fi\fi\fi\catcode'015=5}}\gdef\Address{\obeylines\GetAddress}

\begin{document}
\Title
			Cohomological invariants of genus three hyperelliptic curves
\ShortTitle
			
\SubTitle
			
\Author
		Roberto Pirisi
	
\ShortAuthor
		   
\EndTitle

\Abstract
We compute the cohomological invariants with coefficients in $\mathbb{Z}/p\mathbb{Z}$ of the stack $\mathscr{H}_3$ of hyperelliptic curves of genus $3$ over an algebraically closed field.
\EndAbstract

\MSC
14D43, 14H10 (primary), 14F20, 14D23 (secondary).

\EndMSC

\KEY
Cohomological invariants, hyperelliptic curve, moduli stack, equivariant Chow rings.
\EndKEY

\Address
	Roberto Pirisi
	Department of Mathematics
	KTH Stockholm
	Stockholm
	Sweden
	roberto.pirisi86@gmail.com

\Address
	
\Address
\Address
\EndAddress
\section*{Introduction}

\emph{Notation}: we fix a base field $k_0$ of characteristic different from $2,3$ and a prime number $p\neq \operatorname{char}(k_0)$. All schemes and algebraic stacks will be assumed to be of finite type over $k_0$. If $X$ is a $k_0$-scheme we will denote by $\operatorname{H}^{i}(X)$ the $i$-th \'etale cohomology group of $X$ with coefficients in $\mu_p^{\otimes i}$ (here $\mu_p^{\otimes 0}:=\mathbb{Z}/p\mathbb{Z}$), and by $\operatorname{H}^{\mbox{\tiny{$\bullet$}}}(X)$ the direct sum $\oplus_i \operatorname{H}^i(X)$. If $R$ is a $k_0$-algebra, we set $\operatorname{H}^{\mbox{\tiny{$\bullet$}}}(R)=\operatorname{H}^{\mbox{\tiny{$\bullet$}}}(\operatorname{Spec}(R))$.

\vspace{0.5cm}
Given a smooth algebraic group $\mathrm{G}$, there is a well-known theory of invariants called \emph{cohomological invariants}; examples of cohomological invariants have appeared throughout the literature since the early twentieth century \cite{Wi37}. These were later encompassed in the modern, functorial formulation. The reader can find an introduction to the classical theory of cohomological invariants in the book \cite{GMS03}, by Garibaldi, Merkurjev and Serre. The cohomological invariants of $\mathrm{G}$ form a graded ring $\operatorname{Inv}^{\mbox{\tiny{$\bullet$}}}(\mathrm{G})$. 

In \cite{Pir} the author introduced the concept of cohomological invariant of a smooth algebraic stack. Given a smooth algebraic stack $\mathscr{M}$, we can consider the functor of isomorphisms classes of its points 

\begin{center}
$P_{\mathscr{M}}: \left({\textrm{field}}/{k_0}\right) \rightarrow \left( \textit{set} \right)$
\end{center}which sends a field $K/k_0$ to the isomorphism classes of objects over $K$ in $\mathscr{M}$. Then a cohomological invariant for $\mathscr{M}$ is defined as a natural transformation

\begin{center}
$\alpha: P_{\mathscr{M}} \rightarrow \operatorname{H}^{\mbox{\tiny{$\bullet$}}}(-)$
\end{center}
satisfying a natural continuity condition.

 The cohomological invariants of $\mathscr{M}$ form a graded ring $\operatorname{Inv}^{\mbox{\tiny{$\bullet$}}}(\mathscr{M})$, and when $\mathscr{M}$ is the stack $\mathrm{BG}$ of $\mathrm{G}$-torsors for a smooth algebraic group $\mathrm{G}$, this definition of cohomological invariants retrieves the classical ring of cohomological invariants $\operatorname{Inv}^{\mbox{\tiny{$\bullet$}}}(\mathrm{G})$, that is, we have $$\operatorname{Inv}^{\mbox{\tiny{$\bullet$}}}(\mathrm{G})=\operatorname{Inv}^{\mbox{\tiny{$\bullet$}}}(\mathrm{BG}).$$

The theory set up in \cite{Pir} was used to compute the cohomological invariants of the stacks of hyperelliptic curves of all even genera in \cite{Pir2}. In this paper we compute the cohomological invariants of the stack $\mathscr{H}_3$ of hyperelliptic curves of genus three. The main result is the following:

\begin{theorem}
Suppose our base field $k_0$ is algebraically closed, of characteristic different from $2,3$. For $p=2$ the cohomological invariants of $\mathscr{H}_3$ are freely generated as an $\mathbb{F}_2$-module by $1$ and elements $x_1,x_2,w_2,x_3,x_4,x_5$, where the degree of $x_i$ is $i$ and $w_2$ is the second Stiefel-Whitney class coming from the cohomological invariants of $\mathrm{PGL}_2$.
 
 If $p \neq 2$, then the cohomological invariants of $\mathscr{H}_3$ are trivial for $p \neq 7$ and freely generated by $1$ and a single invariant of degree one for $p=7$.
\end{theorem}

We also get a partial result for general fields, just as in \cite{Pir2}:

\begin{theorem}
Suppose our base field $k_0$ is of characteristic different from $2,3$. For $p=2$ the cohomological invariants of $\mathscr{H}_3$ fit in the exact sequence of $\operatorname{H}^{\mbox{\tiny{$\bullet$}}}(k_0)$-modules

$$0\rightarrow M \rightarrow \operatorname{Inv}^{\mbox{\tiny{$\bullet$}}}(\mathscr{H}_3) \rightarrow K \rightarrow 0$$where $K$ is isomorphic to a submodule of $\operatorname{H}^{\mbox{\tiny{$\bullet$}}}(k_0)$, shifted up in degree by $5$ and $M$ is freely generated as a $\operatorname{H}^{\mbox{\tiny{$\bullet$}}}(k_0)$-module by $1$ and $x_1,x_2,w_2,x_3,x_4$, where the degree of $x_i$ is $i$ and $w_2$ is the second Stiefel-Whitney class coming from the cohomological invariants of $\mathrm{PGL}_2$.

If $p \neq 2$, then the cohomological invariants of $\mathscr{H}_3$ are trivial for $p \neq 7$ and freely generated by $1$ and a single invariant of degree one for $p=7$.
\end{theorem}

The computation heavily uses Rost's theory of Chow groups with coefficients \cite{Rost} and its equivariant version, which was first introduced by Guillot in \cite{Guill}. For a quick introduction to the theory the reader can refer to \cite{Guill} and \cite{Pir2}. The theory of equivariant Chow groups with coefficients is central to the computation due to the fact that for a smooth quotient stack $\left[ X /\mathrm{G} \right]$ the zero-codimensional equivariant Chow group with coefficients $A^0_{\mathrm{G}}(X,\operatorname{H}^{\mbox{\tiny{$\bullet$}}})$ is equal to the ring of cohomological invariants $\operatorname{Inv}^{\mbox{\tiny{$\bullet$}}}(\left[ X/\mathrm{G} \right])$, as proven in \cite[2.10]{Pir2}.

We use the presentation by Vistoli and Arsie \cite[4.7]{VisArs} of the stacks of hyperelliptic curves as the quotient of a smooth scheme by $\mathrm{PGL}_2 \times \mathrm{G}_m$. The stack $\mathscr{H}_3$ is presented as a quotient $\left[ U / \mathrm{PGL}_2 \times \mathrm{G}_m \right]$, where $U$ is an open subscheme of $\mathbb{A}^9$. If we see $\mathbb{A}^{9}$ as the space of binary forms $f=f(x,y)$ of degree $8$, the scheme $U$ is the open subscheme of nonzero forms with distinct roots. 

To compute the cohomological invariants we pass to the projectivized space $Z = U /\mathrm{G}_m$, where $\mathrm{G}_m$ acts by multiplication, and we introduce a stratification $P^{8} \supset \Delta_{1,8} \supset \ldots \supset \Delta_{4,8}$ which will be the base of our computation. We can see $\Delta_{i,8}$ as the closed subscheme of binary forms divisible by the square of a form of degree $i$, and we have $Z = P^{8} \! \smallsetminus \! \Delta_{1,8}$.

The main difference from \cite{Pir2} will be the fact that while for even $g$ the stacks $\mathscr{H}_g$ can be seen as quotients by an action of $\mathrm{GL}_2$, in this case we have to work with the group scheme $\mathrm{PGL}_2 \times \mathrm{G}_m$, which is substantially more complicated. We will need compute the equivariant Chow ring with coefficients $A_{\mathrm{PGL}_2}^{\mbox{\tiny{$\bullet$}}}(\operatorname{Spec}(k_0))$ which turns out to have several nontrivial elements in positive degrees. This poses a challenge, as it is often difficult to understand how these elements behave under pushforward and multiplication. To circumvent this challenge we will use techniques resembling those that the author used for the non-algebraically closed case in \cite[sec.5]{Pir2}.

\renewcommand\thesubsection{\arabic{subsection}}
\section{Some equivariant Chow groups with coefficients}
In this section we will compute some equivariant Chow groups with coefficients which will be needed as a starting point for our computations. The reason for this is the following important equality:

\begin{proposition}\label{inv-chow}
Let $\left[X/\mathrm{G}\right]$ be a quotient stack, smooth over $k_0$. Then $$A^0_\mathrm{G}(X,\operatorname{H}^{\mbox{\tiny{$\bullet$}}})=\operatorname{Inv}^{\mbox{\tiny{$\bullet$}}}(\left[X/\mathrm{G}\right]).$$
\end{proposition}
\begin{proof}
This is proven in \cite[2.10]{Pir2}.
\end{proof} 

 We begin by stating some basic facts about Chow groups with coefficients and their equivariant counterpart. A reader looking for a more in depth introduction to the theory can refer to \cite[sec.2]{Guill} and \cite[sec.1]{Pir2}.
\vspace{0.5cm}

A cycle module $M$ is a functor $M: (\mbox{Fields}/k_0) \rightarrow (\mbox{Groups})$ satisfying a long list of properties, as defined in \cite{Rost}. The two main examples of cycle modules are $M(K)=K_{\mbox{\tiny{$\bullet$}}}$, i.e. Milnor's $K$-theory (in which case the theory is more often referred as $K$-homology), and $M(K)=\operatorname{H}^{\mbox{\tiny{$\bullet$}}}(K)$. In this paper we will always be using the latter.

Let $X$ be an equidimensional scheme. This will always be the case throughout the paper. Define the group $C^i(X,M)$ of $i$-codimensional cycles as $$C^{i}(X,M)=\oplus_{P \in X^{(i)}} M(k(P))$$ where $M$ is a cycle module. Due to the properties of cycle modules there are differential maps $\operatorname{d}: C^i(X,M) \rightarrow C^{i+1}(X,M)$, forming a complex

$$0 \rightarrow C^0(X) \rightarrow C^{1}(X) \rightarrow \ldots \rightarrow C^{\operatorname{dim}(X)}(X) \rightarrow 0.$$

We define the $i$-th Chow group with coefficients $A^{i}(X,M)$ as the $i$-th homology group of the complex above. The group $A^{i}(X,M)$ has a natural double grading. An element $\tau \in C^{i}(X,M)$ is a linear combinations of elements $\alpha \in M(K)$, where $K=k(P)$ for a point $P \in X$. The \emph{codimension} of $\alpha$ is just the index $i$, and it denotes the codimension of $P$ in $X$. Cycle modules are by definition graded modules (or at least $\mathbb{Z}/2\mathbb{Z}$-graded), so we define the \emph{degree} of $\alpha$ to be its degree in $M(K)$. This double grading passes to $A^{i}(X,M)$ as elements in the same equivalence class have the same degree and codimension. 

The subgroup of elements of degree zero can be considered as the ``geometric'' part of the cycle theory; when the cycle module $M$ is equal to $K_{\mbox{\tiny{$\bullet$}}}$, the set of elements of degree zero in $A^{i}(X,K_{\mbox{\tiny{$\bullet$}}})$ is equal to the usual Chow group $\operatorname{CH}^i(X)$, and when $M$ is equal to $\operatorname{H}^{\mbox{\tiny{$\bullet$}}}$, the set of elements of degree zero in $A^{i}(X,\operatorname{H}^{\mbox{\tiny{$\bullet$}}})$ is equal to $\operatorname{CH}^i(X)\otimes_{\mathbb{Z}} \mathbb{Z}/p\mathbb{Z}$, the usual Chow group modulo $p$.

When $X$ is smooth there is a multiplication map sending a couple $(\alpha, \beta)$ of elements of codimension and degree respectively $(i,d),(i',d')$ to an element $\alpha\beta$ of codimension and degree $(i+i', d+d')$. In this case we call the graded ring with unit $A^{\mbox{\tiny{$\bullet$}}}(X,M)=\oplus_{i}A^{i}(X,M)$ the \emph{Chow ring with coefficients} of $X$.

 Given a map $X \xrightarrow{f} Y$ a pullback $f^*$ exists if $Y$ is smooth or $f$ is flat and equidimensional; if $Y$ and $X$ are smooth the pullback is a map of graded rings with unit. A pushforward $f_*$ exists if $f$ is proper, and if $Y$ is smooth the pushforward is a map of $A^{\mbox{\tiny{$\bullet$}}}(Y,M)$-modules. 
 
 Given a closed immersion $V \xrightarrow{i} X$ of codimension $c$, denote by $U$ the complement of $V$. There is a localization exact sequence $$\ldots \rightarrow A^{i-c}(V,M) \xrightarrow{i_*} A^{i}(X,M) \rightarrow A^{i}(U,M) \xrightarrow{\partial} A^{i-c+1}(V,M) \xrightarrow{i_*} \ldots$$  where the boundary map $\partial$ lowers degree by one. Finally, an affine bundle induces an isomorphism on Chow groups with coefficients, and there is a theory of Chern classes satisfying the usual properties.

\vspace{0.5cm}

In the case where $X$ is acted upon by an algebraic group $\mathrm{G}$, we can define an equivariant Chow group with coefficient $A^{i}_\mathrm{G}(X)$ by taking a representation $W$ of $\mathrm{G}$ such that $\mathrm{G}$ acts freely on an open subset $U\subset W$ whose complement has codimension higher than $i+1$. Then $\mathrm{G}$ acts freely on $X \times U$ and we define $$A^{i}_\mathrm{G}(X,M):=A^{i}((X\times U)/\mathrm{G},M)$$ where the action of $\mathrm{G}$ is the diagonal one. One can show that this groups only depend on the isomorphism class of the quotient stack $\left[ X/\mathrm{G} \right]$. When $X$ and $\mathrm{G}$ are smooth we obtain a graded ring with unit $A^{\mbox{\tiny{$\bullet$}}}_{\mathrm{G}}(X,M)= \oplus_{i}A^i_\mathrm{G}(X,M).$ All the properties mentioned above extend to the equivariant case, where instead of any morphism $f: X \rightarrow Y$ we consider only equivariant morphisms.

\vspace{0.5cm}

In the following, the cycle module we use will always be \'etale cohomology, so we will often shorten $A^i(X,\operatorname{H}^{\mbox{\tiny{$\bullet$}}})$ to $A^{i}(X)$, and $A^{i}_{\mathrm{G}}(X,\operatorname{H}^{\mbox{\tiny{$\bullet$}}})$ to $A^{i}_{\mathrm{G}}(X)$. 

Our aim is to compute some equivariant Chow groups with coefficients leading to $A^{\mbox{\tiny{$\bullet$}}}_{\mathrm{SO}_3}(\operatorname{Spec}(k_0),\operatorname{H}^{\mbox{\tiny{$\bullet$}}})$. If we consider the bilinear form 
$$\langle A, B \rangle = \operatorname{tr}(AB)$$
on the space $V$ of two by two matrices of trace zero, the conjugation action by $\mathrm{PGL}_2$ on it preserves it, and it acts with determinant $1$, inducing an isomorphism $\mathrm{PGL}_2 \simeq \mathrm{SO}(Q)$, where $Q(A)=\operatorname{tr}(A^2)$. As the form $Q$ is equivalent to a multiple of the standard form $x_1x_2 + x_3^2$ on $V$, we get an isomorphism $\mathrm{PGL}_2 \simeq \mathrm{SO}_3$, which induces an isomorphism 
$$A^{\mbox{\tiny{$\bullet$}}}_{\mathrm{SO}_3}(\operatorname{Spec}(k_0),\operatorname{H}^{\mbox{\tiny{$\bullet$}}}) \simeq A^{\mbox{\tiny{$\bullet$}}}_{\mathrm{PGL}_2}(\operatorname{Spec}(k_0),\operatorname{H}^{\mbox{\tiny{$\bullet$}}}).$$

 The latter is necessary for our computation as $\mathscr{H}_3$ can be presented as the quotient stack $\left[ U / \mathrm{PGL}_2 \times \mathrm{G}_m \right]$.

Note moreover that the equivariant Chow rings with coefficients $A^{\mbox{\tiny{$\bullet$}}}_{\mathrm{O}(Q)}(\operatorname{Spec}(k_0)$ is the same for all non-degenerate forms $Q$. When we consider the special orthogonal group, the same holds true for all non-degenerate forms with the same discriminant. This is explained in \cite[4.2]{VisMoli} for ordinary equivariant Chow groups. The same argument carries for Chow groups with coefficients. 

\vspace{0.5cm}

We begin by computing $A_{\mu_q}^{\mbox{\tiny{$\bullet$}}}(\operatorname{Spec}(k_0))$, where $q$ is a prime different from the characteristic of $k_0$.

 The cohomological invariants of $\mu_q$, which are equal to $A^{0}_{\mu_q}(\operatorname{Spec}(k_0),\operatorname{H}^{\mbox{\tiny{$\bullet$}}})$, are trivial if $p \neq q$ and are freely generated as an $\operatorname{H}^{\mbox{\tiny{$\bullet$}}}$-module by $1$ and a single invariant $\alpha$ in degree one if $p=q$. Thus $\alpha^2$ is a $\operatorname{H}^{\mbox{\tiny{$\bullet$}}}$-linear combination of $\alpha$ and $1$. More precisely, consider the element $\lbrace -1 \rbrace \in k_0/k_0^p \simeq \operatorname{H}^{1}(k_0)$ which is equal to $0$ except possibly when $p=2$. We have $\alpha^2 = \lbrace -1 \rbrace \cdot \alpha$.
 
 The ordinary Chow ring $\mathrm{CH}_{\mu_q}(\operatorname{Spec}(k_0))$ is generated as a $\mathbb{Z}$-algebra by $1$ and a single element $\xi$ of $q$-torsion, corresponding to the first Chern class of the vector bundle obtained from the representation $\mu_q \subset \mathbb{G}_m \curvearrowright \mathbb{A}^1$.
 
\begin{proposition}
Let $k$ be a field and $q$ be a prime different from the characteristic of $k_0$. 

\begin{itemize}
\item If $q\neq p$, then  $A^{\mbox{\tiny{$\bullet$}}}_{\mu_q}(\operatorname{Spec}(k_0),\operatorname{H}^{\mbox{\tiny{$\bullet$}}})$  is equal to $\operatorname{H}^{\mbox{\tiny{$\bullet$}}}(k_0)$, that is, it is generated by $1$ as a free $\operatorname{H}^{\mbox{\tiny{$\bullet$}}}(k_0)$-module.
\item If $p=q$, then  $A^{\mbox{\tiny{$\bullet$}}}_{\mu_q}(\operatorname{Spec}(k_0),\operatorname{H}^{\mbox{\tiny{$\bullet$}}})$ is $\operatorname{H}^{\mbox{\tiny{$\bullet$}}}(\operatorname{Spec}(k_0))\left[\alpha,\xi \right]/(\alpha^2-\lbrace -1 \rbrace \alpha)$. 
\end{itemize}

The element $\xi$ has codimension one and degree zero, and it comes from the ordinary Chow ring. The element $\alpha$ is an element in codimension zero and degree one, corresponding to a generator for the cohomological invariants of $\mu_q$.

\end{proposition}
\begin{proof}
We consider the action of $\mu_q$ on $\mathrm{G}_m$ induced by the inclusion. This action extends linearly to $\mathbb{A}^1_k$. Then there is a long exact sequence:

\begin{center}
$ 0 \rightarrow A^0_{\mu_q}(\mathbb{A}^1_{k_0}) \rightarrow A^0_{\mu_q}(\mathrm{G}_m) \xrightarrow{\partial} A^0_{\mu_q}(\operatorname{Spec}(k_0)) \xrightarrow{c_1}   A^1_{\mu_q}(\mathbb{A}^1_{k_0}) \rightarrow \ldots $
\end{center}

Using the retraction $r$ described in \cite[section 9]{Rost} we identify $A^{\mbox{\tiny{$\bullet$}}}_{\mu_q}(\mathbb{A}^1_{k_0})$ with $A^{\mbox{\tiny{$\bullet$}}}_{\mu_q}(\operatorname{Spec}(k_0))$ and consequently the inclusion pushforward with the first Chern class, $c_1$, for the equivariant vector bundle $\mathbb{A}^1_k \rightarrow \operatorname{Spec}(k_0)$. As all the stacks here are smooth we have that the map $c_1$ is equal to multiplication by an element $\xi$ of degree zero and codimension $1$. Note now that $\left[ \mathrm{G}_m/\mu_q \right] \simeq \mathrm{G}_m$, so that $A^{\mbox{\tiny{$\bullet$}}}_{\mu_q}(\mathrm{G}_m) =A^{\mbox{\tiny{$\bullet$}}}(\mathrm{G}_m)$ which is equal to $$\operatorname{H}^{\mbox{\tiny{$\bullet$}}}(k_0) \oplus \operatorname{H}^{\mbox{\tiny{$\bullet$}}}(k_0)  \alpha$$ by \cite[2.1.1]{Guill}, where $\alpha$ is an element in codimension zero and degree one. The boundary map $\partial$ applied to this element is equal to $q$, which shows that $q \xi =0$. The computation immediately follows as $A^{i}(\mathrm{G}_m)$ is zero for $i > 0$, which shows that multiplication by $\xi$ is an isomorphism $A^{i}_{\mu_q}(\operatorname{Spec}(k_0)) \rightarrow A^{i+1}_{\mu_q}(\operatorname{Spec}(k_0))$ for each $i \geq 1$ when $p=q$, and it is always zero when $p \neq q$.
\end{proof}

The reasoning works the same for an algebraic space being acted on trivially by $\mu_q$.

\begin{lemma}\label{mu_q}
Let $X$ be an algebraic space over a field $k$, and let $\mu_q$ act trivially on it. Then $A^{\mbox{\tiny{$\bullet$}}}_{\mu_q}(X)=A^{\mbox{\tiny{$\bullet$}}}(X)\otimes_{\operatorname{H}^{\mbox{\tiny{$\bullet$}}}(k_0)} A^{\mbox{\tiny{$\bullet$}}}_{\mu_q}(\operatorname{Spec}(k_0))$.
\end{lemma}
\begin{proof}
We consider again the exact sequence:
\begin{center}
$ 0 \rightarrow A^0_{\mu_q}(X \times \mathbb{A}^1) \xrightarrow{j^*} A^0_{\mu_q}(X \times \mathrm{G}_m) \xrightarrow{\partial} A^0_{\mu_q}(X) \xrightarrow{c_1}  A^1_{\mu_q}(X \times \mathbb{A}^1) \rightarrow \ldots $
\end{center}
As before, the quotient $\left[ (X \times \mathrm{G}_m)/\mu_q \right]$ is isomorphic to $X\times \mathrm{G}_m$, so that for its Chow groups with coefficients the formula $A^i_{\mu_q}(X \times \mathrm{G}_m) = A^i(X) \oplus A^i(X) \alpha$ holds.

 As the first component comes from the pullback through $X \times \mathrm{G}_m \rightarrow X$ and this map factors through $\left[ (X\times \mathbb{A}^1)/\mu_q \right]$ we see that the first component always belongs to the image of $j^*$, and given an element $t \cdot \alpha$ in the second component its image through the boundary map $\partial$ is equal to $q$ times $ t$. This gives us a complete understanding of the exact sequence, allowing us to conclude the proof of lemma \ref{mu_q}.
 \end{proof}
 
This also works when we have a space being acted on $G \times \mu_q$, and the action of $\mu_q$ is trivial. To prove it we need a lemma.

\begin{lemma}\label{eq1}
Let $\mathrm{G}$ be a linear algebraic group, acting on an algebraic space $X$ smooth over $k_0$, and let $\mathrm{H}$ be a normal subgroup of $\mathrm{G}$. Suppose the action of $\mathrm{H}$ on $X$ is free with quotient $X/\mathrm{H}$. Then there is a canonical isomorphism
$$ A^{\mbox{\tiny{$\bullet$}}}_\mathrm{G}(X) \simeq A^{\mbox{\tiny{$\bullet$}}}_{\mathrm{G}/\mathrm{H}} (X/\mathrm{H}).$$
\end{lemma}
\begin{proof}
The proof in \cite{VisMoli}[2.1] works without any change.
\end{proof}

\begin{corollary}\label{mu_q2}
Let $X$ be an algebraic space over a field $k$, and let $\mathrm{G}$ be an affine group acting on it. Let $\mathrm{G} \times \mu_q$ act on $X$ through the first projection $\mathrm{G} \times \mu_q \rightarrow \mathrm{G}$. Then $$A^{\mbox{\tiny{$\bullet$}}}_{\mathrm{G} \times \mu_q}(X)=A^{\mbox{\tiny{$\bullet$}}}_{\mathrm{G}}(X)\otimes_{\operatorname{H}^{\mbox{\tiny{$\bullet$}}}(k_0)} A^{\mbox{\tiny{$\bullet$}}}_{\mu_q}(\operatorname{Spec}(k_0)).$$
\end{corollary}
\begin{proof}
It is well known that any affine algebraic group $\mathrm{G}$ is linear and thus it has a generically free representation $W$. By taking powers of $W$ and having $\mathrm{G}$ act diagonally we get a representation $V$ where $\mathrm{G}$ acts freely on an open subset $U$ whose complement has codimension $d$ for any $d$. We extend the action on $X \times V$ to a $\mathrm{G}\times \mu_q$ action via the first projection. Note that the map $X \times V \rightarrow X$ is a $\mathrm{G}\times \mu_q$ equivariant vector bundle, so $A^{\mbox{\tiny{$\bullet$}}}_{\mathrm{G}\times \mu_q}(X) \simeq A^{\mbox{\tiny{$\bullet$}}}_{\mathrm{G}\times \mu_q}(X\times V)$ and thus $A^{i}_{\mathrm{G}\times \mu_q}(X) \simeq A^{i}_{\mathrm{G}\times \mu_q}(X\times U)$ for all $i \leq d$. Then by lemma (\ref{eq1}), where the normal group is $\mathrm{G}$ we get $\simeq A^{i}_{\mathrm{G}\times \mu_q}(X\times U)=A^{i}_{\mu_q}(X\times U/\mathrm{G})$. But the action of $\mu_q$ is trivial, so we get $$A^{i}_{\mu_q}(X\times U/\mathrm{G})=A^{i}(X\times U/\mathrm{G})\otimes_{\operatorname{H}^{\mbox{\tiny{$\bullet$}}}(k_0)} A^{\mbox{\tiny{$\bullet$}}}_{\mu_q}(\operatorname{Spec}(k_0))=$$ $$=A^{i}_{\mathrm{G}}(X)\otimes_{\operatorname{H}^{\mbox{\tiny{$\bullet$}}}(k_0)} A^{\mbox{\tiny{$\bullet$}}}_{\mu_q}(\operatorname{Spec}(k_0))$$ concluding the proof.
\end{proof}

We can now compute the equivariant Chow ring  $A^{\mbox{\tiny{$\bullet$}}}_{\mathrm{O}_n}(\operatorname{Spec}(k_0))$ for $n = 2, 3$ with coefficients in $\operatorname{H}^{\mbox{\tiny{$\bullet$}}}$. This should serve as an example of how the Chow groups with coefficients can start behaving wildly even for well known objects, as elements of positive degree with no clear geometric or cohomological description appear. 

We will follow the method in \cite[4.1]{VisMoli}. First we need a few more lemmas, which are by themselves interesting facts about the equivariant approach. We begin by explicitly identifying a class of algebraic groups having the property that under specific conditions they can be ignored while computing equivariant Chow groups with coefficients. This was done in the case of ordinary equivariant Chow groups by Vistoli and Molina.

\begin{definition}\label{star}

Let $\mathrm{H}$ be a linear algebraic group. We say that $\mathrm{H}$ has the property $(*)$ if there is an isomorphism $\phi: \mathrm{H} \simeq \mathbb{A}^n_k $ of varieties such that for any field extension $k' \supseteq k$ and any element $h \in \mathrm{H}(k')$ the automorphism of $\mathbb{A}^n_k$ corresponding through $\phi$ to the action of $h$ on $\mathrm{H}_k$ by left multiplication is affine (i.e. a composition of a linear maps and a translation).
\end{definition}

A more abstract way to state the definition above is the following. Let $V$ be a finite dimensional vector space and let $\textrm{Aff}(V)$ be the semi-direct product $V\rtimes \mathrm{GL}(V)$ viewed as the algebraic group of affine transformations of $V$. Let $p: \textrm{Aff}(V) \rightarrow V$ be the projection (which is not a group homomorphsim).

Then a linear algebraic group $\mathrm{H}$ has the property $(*)$ if $\mathrm{H}$ can be embedded as a subgroup of $\textrm{Aff}(V)$ for some $V$ and additionally the composition with the projection $p$ is an isomorphism $\phi: \mathrm{H} \xrightarrow{\simeq} V$ of algebraic varieties.

\begin{lemma}\label{eq2} 
Let $\mathrm{H}$ be an a linear algebraic group satisfying property $(*)$, and let $\mathrm{G}$ be a linear algebraic group acting on $\mathrm{H}$ via group automorphisms, corresponding to linear automorphisms of $\mathbb{A}^n_k$ under $\phi$.

If $\mathrm{G}$ acts on an algebraic space $X$ smooth over $k_0$, form the semidirect product $\mathrm{G}\ltimes \mathrm{H}$ and let it act on $X$ via the projection $\mathrm{G}\ltimes \mathrm{H} \rightarrow \mathrm{G}$. Then the homomorphism 
$$ A^{\mbox{\tiny{$\bullet$}}}_\mathrm{G}(X) \rightarrow A^{\mbox{\tiny{$\bullet$}}}_{\mathrm{G}\ltimes \mathrm{H}}(X)$$
induced by the projection $\mathrm{G}\ltimes \mathrm{H} \rightarrow \mathrm{G}$ is an isomorphism.
\end{lemma}
\begin{proof}
Again the argument used in \cite[2.3]{VisMoli} works for any equivariant theory defined as in \cite{EdiGra}.
\end{proof}

Recall now that when $p=2$, the ring $A^0_{\mathrm{O}_n}(\operatorname{Spec}(k_0),\operatorname{H}^{\mbox{\tiny{$\bullet$}}})=\operatorname{Inv}^{\mbox{\tiny{$\bullet$}}}(B\mathrm{O}_n)$ is freely generated as a $\operatorname{H}^{\mbox{\tiny{$\bullet$}}}(k_0)$-module by the Steifel-Whitney classes $1=w_o, w_1, \ldots, w_n $, where $w_i$ has degree $i$. This is proven in \cite{GMS03}. Moreover, the ordinary $\mathrm{O}_n$-equivariant Chow ring of a point is

$$\operatorname{CH}_{\mathrm{O}_n}^{\mbox{\tiny{$\bullet$}}}(\operatorname{Spec}(k_0))= {\mathbb{Z}\left[ c_1,\ldots, c_n\right]}/{(2c_i)_{( \textit{i odd})}} $$
Where $c_1, \ldots, c_n$ are the Chern classes of the standard representation of $\mathrm{O}_n$.

We will adjust the argument from \cite[4.1]{VisMoli}, which computes the ordinary equivariant Chow groups. Let $q$ be standard quadratic form given by $$q(x)=x_1x_{m+1} + x_2x_{m+2} +\ldots + x_{m}x_{2m}$$ when $n=2m$ and $$q(x)=x_1x_{m+1} + x_2x_{m+2} +\ldots + x_{m}x_{2m}+x_{2m+1}^2$$ when $n=2m+1$, fixed by $\mathrm{O}_n=\mathrm{O}(q)$. We begin with some general consideration before tackling the specifics of the $n=2, n=3$ cases.

\vspace{0.5cm}

 Let $V$ be the standard $n$-dimensional representation of $\mathrm{O}_n$. We want to compute $A^{\mbox{\tiny{$\bullet$}}}_{\mathrm{O}_n}(V)=A^{\mbox{\tiny{$\bullet$}}}_{\mathrm{O}_n}(\operatorname{Spec}(k_0))$. We will stratify $V$ as the union of $B=\lbrace q \neq 0 \rbrace, C=\lbrace q=0 \rbrace \! \smallsetminus \! \lbrace 0 \rbrace$ and the origin $\lbrace 0 \rbrace$.

The map $q: B \rightarrow \mathrm{G}_m$ can be trivialized by passing to the \'etale covering $\tilde{B}=\lbrace (t,v) \in \mathrm{G}_m \times B \mid t^2=q(v) \rbrace$, with $\mu_2$ acting by multiplication on the left component. We have $\tilde{B}/\mu_2=B$. Let $Q$ denote the locus where $q=1$. Then $\tilde{B}$ is isomorphic to $ \mathrm{G}_m\times Q$, the action of $\mu_2$ is the multiplication on both components and the action of $\mathrm{O}_n$ is the action on the second component. The $\mathrm{G}_m$-torsor $$\left[\tilde{B}/\mathrm{O}_n \times \mu_2 \right]=\left[ B/\mathrm{O}_n \right] \rightarrow \left[ Q/\mathrm{O}_n \times \mu_2 \right]$$ can be completed to a line bundle $\mathcal{E} \rightarrow \left[ Q/\mathrm{O}_n \times \mu_2\right]$, which corresponds to an $\mathrm{O}_n\times \mu_2$-equivariant line bundle on $Q$, so that the inclusion of the zero section gives rise to a long exact sequence

$$\ldots A^i_{\mathrm{O}_n \times \mu_2 }(Q) \rightarrow A^i_{\mathrm{O}_n \times \mu_2}(\tilde{B}) \xrightarrow{\partial}  A^i_{\mathrm{O}_n \times \mu_2 }(Q) \xrightarrow{c_1}  A^{i+1}_{\mathrm{O}_n \times \mu_2}(Q) \ldots$$
where we are identifying $A^{\mbox{\tiny{$\bullet$}}}_{\mathrm{O}_n \times \mu_2}(\mathcal{E})$ with $A^{\mbox{\tiny{$\bullet$}}}_{\mathrm{O}_n \times \mu_2 }(Q)$, which in turn identifies the pushforward through the zero section with $c_1(\mathcal{E})$.

We can see as in \cite[pp.283-284]{VisMoli} that $\mathrm{O}_n\times \mu_2$ acts transitively on $Q$ with stabilizer $\mathrm{O}_{n-1}\times \mu_2$, so we have $$A^{\mbox{\tiny{$\bullet$}}}_{\mathrm{O}_n\times \mu_2}(Q)=A^{\mbox{\tiny{$\bullet$}}}_{\mathrm{O}_{n-1}\times \mu_2}(\operatorname{Spec}(k_0)).$$

 We can now use corollary (\ref{mu_q2}). In the case of $p=2$ we get 
 $$A^{\mbox{\tiny{$\bullet$}}}_{\mathrm{O}_n\times \mu_2}(Q)=A^{\mbox{\tiny{$\bullet$}}}_{\mathrm{O}_{n-1}\times \mu_2}(\operatorname{Spec}(k_0))=A^{\mbox{\tiny{$\bullet$}}}_{\mathrm{O}_{n-1}}(\operatorname{Spec}(k_0))\left[\xi,\alpha\right]/(\alpha^2-\lbrace -1 \rbrace \alpha).$$

  When $p \neq 2$ we get $$A^{\mbox{\tiny{$\bullet$}}}_{\mathrm{O}_n\times \mu_2}(Q)=A^{\mbox{\tiny{$\bullet$}}}_{\mathrm{O}_{n-1}\times \mu_2}(\operatorname{Spec}(k_0))=A^{\mbox{\tiny{$\bullet$}}}_{\mathrm{O}_{n-1}}(\operatorname{Spec}(k_0)).$$
    The class $c_1(\mathcal{E})$ is equal to $\xi$, as shown in \cite[p.284]{VisMoli}. When $M=\operatorname{H}^{\mbox{\tiny{$\bullet$}}}$ and $p=2$ multiplication by $\xi$ is injective, so we see that $$A^{\mbox{\tiny{$\bullet$}}}_{\mathrm{O}_n}(B)= A^{\mbox{\tiny{$\bullet$}}}_{\mathrm{O}_n\times \mu_2}(\tilde{B})= A^{\mbox{\tiny{$\bullet$}}}_{\mathrm{O}_n\times \mu_2}(Q) /c_1(\mathcal{E}).$$and thus $$A^{\mbox{\tiny{$\bullet$}}}_{\mathrm{O}_n}(B)=A^{\mbox{\tiny{$\bullet$}}}_{\mathrm{O}_{n-1}}(\operatorname{Spec}(k_0))\oplus A^{\mbox{\tiny{$\bullet$}}}_{\mathrm{O}_{n-1}}(\operatorname{Spec}(k_0))\cdot \alpha.$$

 In the case $ p \neq 2 $ we no longer have the element $\alpha$ in $A^{\mbox{\tiny{$\bullet$}}}_{\mathrm{O}_n\times \mu_2}(Q)$ but the map $c_1$ is trivial as $2 \xi=0$ and $2$ is invertible, so we get again $$A^{\mbox{\tiny{$\bullet$}}}_{\mathrm{O}_n}(B)=A^{\mbox{\tiny{$\bullet$}}}_{\mathrm{O}_{n-1}}(\operatorname{Spec}(k_0)) \oplus A^{\mbox{\tiny{$\bullet$}}}_{\mathrm{O}_{n-1}}(\operatorname{Spec}(k_0))\cdot \alpha'$$ for an element $\alpha'$ in codimension zero and degree one.

Finally, $\mathrm{O}_n$ acts transitively on $C$ with stabilizer a semidirect product of $\mathrm{O}_{n-2}$ and an algebraic group satisfying the $(*)$ property of definition (\ref{star}) by \cite[p.283]{VisMoli}, so that using lemmas (\ref{eq1},\ref{eq2}) we get $A^{\mbox{\tiny{$\bullet$}}}_{\mathrm{O}_n}(C)=A^{\mbox{\tiny{$\bullet$}}}_{\mathrm{O}_{n-2}}(\operatorname{Spec}(k_0))$. Note that when $n=2$ we have $\mathrm{O}_{n-2}=\mathrm{O}_0= \lbrace 1 \rbrace$.

With this, we are ready to tackle the cases $n=2,3$.

\begin{proposition}

Suppose that $p=2$, then the Chow ring with coefficients $A^{\mbox{\tiny{$\bullet$}}}_{\mathrm{O}_2}(\operatorname{Spec}(k_0),\operatorname{H}^{\mbox{\tiny{$\bullet$}}})$ is isomorphic to
$$ A^0_{\mathrm{O}_2}(\operatorname{Spec}(k_0))\left[ c_1, c_2 \right] \oplus \operatorname{H}^{\mbox{\tiny{$\bullet$}}}(k_0)\left[ c_1,c_2 \right] \tau_{1,1}$$
Where $\tau_{1,1}$ is an element of codimension and degree $(1,1)$. The classes $c_i$ are the Chern classes of the standard representation, and the notation $\operatorname{H}^{\mbox{\tiny{$\bullet$}}}(k_0)\left[ c_1,c_2 \right] \tau_{1,1}$ means the free module generated by $\tau_{1,1}$ over the polynomial ring $\operatorname{H}^{\mbox{\tiny{$\bullet$}}}(k_0)\left[ c_1,c_2 \right]$.

Suppose that $p \neq 2$, then $A^{\mbox{\tiny{$\bullet$}}}_{\mathrm{O}_2}(\operatorname{Spec}(k_0),\operatorname{H}^{\mbox{\tiny{$\bullet$}}})$ is the tensor product of $\operatorname{H}^{\mbox{\tiny{$\bullet$}}}(k_0)$ with the ordinary equivariant Chow ring.

\end{proposition}
\begin{proof}
We'll prove the case of $p=2$. The case $p \neq 2$ can be easily done in the same way, as the same exact sequences hold.

 We already know the rings $A^{\mbox{\tiny{$\bullet$}}}_{\mathrm{O}_n}(\operatorname{Spec}(k_0))$ for $n=0,1$, all that remains is to understand the long exact sequences coming from the equivariant inclusions $C \rightarrow V \! \smallsetminus \! \lbrace 0 \rbrace$ and $\lbrace 0 \rbrace \rightarrow V$.

For $n=2$ we know that the ring $A^{\mbox{\tiny{$\bullet$}}}_{\mathrm{O}_2}(C)$ is equal to $M(k_0)$ and that the pushforward $A^{\mbox{\tiny{$\bullet$}}}_{\mathrm{O}_2}(C) \rightarrow A^{\mbox{\tiny{$\bullet$}}}_{\mathrm{O}_2}(V \! \smallsetminus \! \lbrace 0 \rbrace)$ must map it to zero as in \cite[p.285]{VisMoli} due to the projection formula, so that we get the exact sequence

$$0 \rightarrow A^i_{\mathrm{O}_2}(V \! \smallsetminus \! \lbrace 0 \rbrace) \rightarrow A^i_{\mathrm{O}_2}(B) \xrightarrow{\partial} A^i_{\mathrm{O}_2}(C) \rightarrow 0$$

 The surjectivity of the map $\partial$ forces the boundary $\partial(\alpha)$ of the element $\alpha \in A^0_{\mathrm{O}_2}(B)$ to be equal to $1$. As the map $A^{\mbox{\tiny{$\bullet$}}}_{\mathrm{O}_2}(V\! \smallsetminus \! \lbrace 0 \rbrace) \rightarrow A^{\mbox{\tiny{$\bullet$}}}_{\mathrm{O}_2}(B)$ is injective, we have $$A^{\mbox{\tiny{$\bullet$}}}_{\mathrm{O}_2}(V\! \smallsetminus \! \lbrace 0 \rbrace)=A^{\mbox{\tiny{$\bullet$}}}_{\mu_2}(\operatorname{Spec}(k_0)) \oplus \operatorname{H}^{\mbox{\tiny{$\bullet$}}}(k_0)\left[ c_1 \right] \tilde{\tau}_{1,1} \oplus \operatorname{H}^{\mbox{\tiny{$\bullet$}}}(k_0)\beta\left[ c_1 \right]$$ where $\tilde{\tau}_{1,1}$ is an element in degree and codimension $1$,and $\beta$ is an element in codimension $0$ and degree $2$, that is $$A^{\mbox{\tiny{$\bullet$}}}_{\mathrm{O}_2}(V\! \smallsetminus \! \lbrace 0 \rbrace)=A^{0}_{\mathrm{O}_2}(\operatorname{Spec}(k_0))\left[c_1\right] \oplus \operatorname{H}^{\mbox{\tiny{$\bullet$}}}(k_0)\tau_{1,1}\left[c_1\right].$$

Observe now that the map $A^{\mbox{\tiny{$\bullet$}}}_{\mathrm{O}_2}(V) \rightarrow A^{\mbox{\tiny{$\bullet$}}}_{\mathrm{O}_2}(V\! \smallsetminus \! \lbrace 0 \rbrace)$ is a map of rings and it is surjective in codimension $0$ (as $\lbrace 0 \rbrace$ has codimension $2$) and in degree $0$ (by \cite[pp.285-286]{VisMoli}) for all codimensions; consider the exact sequence induced by the inclusion $\lbrace 0 \rbrace \rightarrow V$
$$ \ldots \rightarrow A^i_{\mathrm{O}_2}(V) \xrightarrow{j} A^i_{\mathrm{O}_2}(V\! \smallsetminus \! \lbrace 0 \rbrace) \xrightarrow{\partial} A^{i-1}_{\mathrm{O}_2}(\lbrace 0 \rbrace) \xrightarrow{c_2}  A^{i+1}_{\mathrm{O}_2}(V) \rightarrow \ldots $$ where the map $c_2$ is the second Chern class $c_2(V)$. We can see that $\tau_{1,1}$ must be in the image of $j:A^i_{\mathrm{O}_2}(V) \rightarrow A^i_{\mathrm{O}_2}(V\! \smallsetminus \! \lbrace 0 \rbrace)$ as the second Chern class $c_2$ is injective in degree zero, so we must have $\partial(\tau_{1,1})=0$. Then the map of rings $A^{\mbox{\tiny{$\bullet$}}}_{\mathrm{O}_2}(V) \rightarrow A^{\mbox{\tiny{$\bullet$}}}_{\mathrm{O}_2}(V\! \smallsetminus \! \lbrace 0 \rbrace)$ must be surjective, as all generators of $A^{\mbox{\tiny{$\bullet$}}}_{\mathrm{O}_2}(V\! \smallsetminus \! \lbrace 0 \rbrace)$ belong to the image. Thus we get the exact sequence
 
$$0 \rightarrow A^{i-1}_{\mathrm{O}_2}(\lbrace 0 \rbrace) \xrightarrow{c_2} A^{i+1}_{\mathrm{O}_2}(V) \rightarrow A^{i+1}_{\mathrm{O}_2}(V\! \smallsetminus \! \lbrace 0 \rbrace) \rightarrow 0.$$
The exact sequence tells us that multiplication by the second Chern class $c_2$ is injective in $A^{\mbox{\tiny{$\bullet$}}}_{\mathrm{O}_2}(\operatorname{Spec}(k_0))$ and that the quotient by the ideal generated by $c_2$ is equal to $A^{\mbox{\tiny{$\bullet$}}}_{\mathrm{O}_2}(V\! \smallsetminus \! \lbrace 0 \rbrace)$. Then the ring $A^{\mbox{\tiny{$\bullet$}}}_{\mathrm{O}_2}(\operatorname{Spec}(k_0))$ is generated by the generators of $A^{\mbox{\tiny{$\bullet$}}}_{\mathrm{O}_2}(V\! \smallsetminus \! \lbrace 0 \rbrace)$ and $c_2$, and concluding the proof is an easy computation.
\end{proof}

\begin{proposition}
Suppose $p=2$. 
 We have
$$ A^{\mbox{\tiny{$\bullet$}}}_{\mathrm{O}_3}(\operatorname{Spec}(k_0),\operatorname{H}^{\mbox{\tiny{$\bullet$}}}) = A^0_{\mathrm{O}_3}(\operatorname{Spec}(k_0))\left[ c_1,c_2 ,c_3 \right] \oplus \operatorname{H}(k_0)\left[ c_1, c_2, c_3 \right] \tau_{1,1}$$
where again $\tau_{1,1}$ is an element of codimension and degree $(1,1)$.

Suppose $p \neq 2$. Then $A^{\mbox{\tiny{$\bullet$}}}_{\mathrm{O}_3}(\operatorname{Spec}(k_0),\operatorname{H}^{\mbox{\tiny{$\bullet$}}})$ is equal to the tensor product of $\operatorname{H}^{\mbox{\tiny{$\bullet$}}}(k_0)$ with the ordinary equivariant Chow ring.
\end{proposition}
\begin{proof}
We prove the case $p=2$. The case $p \neq 2$ is much easier and can be proven using the same arguments, as the same exact sequences hold.

 For $n=3$, we need to consider the same exact sequences as above. First we have the one coming from the inclusion $C \rightarrow V_3 \! \smallsetminus \! \lbrace 0 \rbrace$:

$$ \ldots \rightarrow A^i_{\mathrm{O}_3}(V\! \smallsetminus \! \lbrace 0 \rbrace) \rightarrow A^i_{\mathrm{O}_3}(B) \xrightarrow{\partial} A^{i}_{\mathrm{O}_3}(C) \rightarrow A^{i+1}_{\mathrm{O}_3}(V\! \smallsetminus \! \lbrace 0 \rbrace) \rightarrow \ldots .$$

The map $A^{i}_{\mathrm{O}_3}(C) \rightarrow A^{i+1}_{\mathrm{O}_3}(V\! \smallsetminus \! \lbrace 0 \rbrace)$ is zero on ordinary Chow groups by \cite{VisMoli}, and we have $A^{i}_{\mathrm{O}_3}(C) \simeq A^{i}_{\mu_2}(\operatorname{Spec}(k_0))$, so we only have to prove that the generator for the cohomological invariants of $\mu_2$ goes to zero.
 To see that, note that $A^0_{\mathrm{O}_3}(V\! \smallsetminus \! \lbrace 0 \rbrace)$ is isomorphic to $A^0_{\mathrm{O}_3}(V)$ which is in turn equal to $\operatorname{Inv}(\mathrm{O}_3)$. So it is a free $\operatorname{H}^{\mbox{\tiny{$\bullet$}}}(\operatorname{Spec}(k_0))$-module of rank three, generated by the Stiefel-Whitney classes $w_1,w_2,w_3$, of degree respectively $1,2,3$.
 
  On the other hand, $A^0_{\mathrm{O}_3}(B)\simeq A^0_{\mathrm{O}_2}(\operatorname{Spec}(k_0))\oplus  A^0_{\mathrm{O}_2}(\operatorname{Spec}(k_0)) \alpha$ is generated as a free $\operatorname{H}^{\mbox{\tiny{$\bullet$}}}(\operatorname{Spec}(k_0))$-module by $w_1,\alpha,w_1 \alpha, w_2,w_2 \alpha$.
  Then the cokernel of the restriction map induced by $B \rightarrow V \!\smallsetminus \! \lbrace 0 \rbrace$ must contain a free $\operatorname{H}^{\mbox{\tiny{$\bullet$}}}(\operatorname{Spec}(k_0))$-module generated by an element in degree two. The boundary map $\partial$ must send it to a generator for the cohomological invariants of $\mu_2$ as it is the only element of degree one in there. This shows that the pushforward $ A^{i}_{\mathrm{O}_3}(C) \rightarrow A^{i+1}_{\mathrm{O}_3}(V\! \smallsetminus \! \lbrace 0 \rbrace)$ is zero, so we have the exact sequence

$$0 \rightarrow A^i_{\mathrm{O}_3}(V \! \smallsetminus \! \lbrace 0 \rbrace) \rightarrow A^i_{\mathrm{O}_3}(B) \rightarrow A^{i}_{\mathrm{O}_3}(C) \rightarrow 0.$$which tells us that $A^{\mbox{\tiny{$\bullet$}}}_{\mathrm{O}_3}(V \! \smallsetminus \! \lbrace 0 \rbrace) \simeq A^0_{\mathrm{O}_3}(\operatorname{Spec}(k_0),\operatorname{H}^{\mbox{\tiny{$\bullet$}}})\left[c_1,c_2 \right] \oplus \operatorname{H}^{\mbox{\tiny{$\bullet$}}}(k_0)\tau_{1,1}\left[ c_1, c_2 \right]$.

Now we consider the last exact sequence. As before, the map of rings $A^{\mbox{\tiny{$\bullet$}}}_{\mathrm{O}_3}(V) \rightarrow A^{\mbox{\tiny{$\bullet$}}}_{\mathrm{O}_3}(V \! \smallsetminus \! \lbrace 0 \rbrace)$ must be surjective. We know that it is surjective in degree $0$ by \cite[pp.285-286]{VisMoli}, and it induces an isomorphism in codimension $1$ and $2$. Then we have the exact sequence

$$0 \rightarrow A^{i-2}_{\mathrm{O}_3}(\lbrace 0 \rbrace) \xrightarrow{c_3} A^{i+1}_{\mathrm{O}_3}(V) \rightarrow A^{i+1}_{\mathrm{O}_3}(V\! \smallsetminus \! \lbrace 0 \rbrace) \rightarrow 0$$which again shows that the ring $A^{\mbox{\tiny{$\bullet$}}}_{\mathrm{O}_3}(V)$ is generated by $A^{\mbox{\tiny{$\bullet$}}}_{\mathrm{O}_3}(V \! \smallsetminus \lbrace 0 \rbrace)$ and $c_3$, and that multiplication by $c_3$ is injective. Using this a simple computation allows us to conclude. 

\end{proof}

\begin{corollary}\label{SO3}
Suppose $p=2$. We have $$A^{\mbox{\tiny{$\bullet$}}}_{\mathrm{SO}_3}(\operatorname{Spec}(k_0),\operatorname{H}^{\mbox{\tiny{$\bullet$}}})=A^0_{\mathrm{SO}_3}(\operatorname{Spec}(k_0),\operatorname{H}^{\mbox{\tiny{$\bullet$}}})\left[ c_2, c_3 \right] \oplus\operatorname{H}^{\mbox{\tiny{$\bullet$}}}(k_0)\left[ c_2, c_3 \right] \tau_{1,1}.$$

Suppose $p\neq 2$. Then $$A^{\mbox{\tiny{$\bullet$}}}_{\mathrm{SO}_3}(\operatorname{Spec}(k_0),\operatorname{H}^{\mbox{\tiny{$\bullet$}}})=\operatorname{H}^{\mbox{\tiny{$\bullet$}}}(k_0)\otimes \operatorname{CH}^{\mbox{\tiny{$\bullet$}}}_{\mathrm{SO}_3}(\operatorname{Spec}(k_0)).$$

\end{corollary}
\begin{proof}
It suffices to use the fact that $\mathrm{O}_3= \mu_2 \times \mathrm{SO}_3 $ and apply lemma (\ref{mu_q}).
\end{proof}

 \section{Preliminaries}

In this section we recall the presentations of the stacks we will work with, all due to Vistoli and Arsie \cite{VisArs}. We will then lay down some lemmas that will be needed for the final computation.

\begin{theorem}\label{H3_quot}

Consider $\mathbb{A}^{9}$ as the space of all binary forms of degree $8$. Denote by $X$ the open subset consisting of nonzero forms with distinct roots, and let $\textrm{PGL}_2 \times \mathrm{G}_m$ act on it by $(\left[ A \right],\alpha)(f)(x)= \operatorname{Det}(A)^{4}\alpha^{-2}f(A^{-1}(x))$.

Then for the stack $\mathscr{H}_{3}$ of smooth hyperelliptic curves of genus $3$ we have
  $$\mathscr{H}_{3}=\left[ X/(\mathrm{PGL}_2 \times \mathrm{G}_m) \right].$$
  
In general the same construction gives us $$\mathscr{H}_{g}=\left[ X_g/(\mathrm{PGL}_2 \times \mathrm{G}_m)\right]$$ where $X_g$ is the open subscheme of $\mathbb{A}^{2g+3}$ parametrizing forms of degree $2g+2$ with distinct roots. 
\end{theorem}
\begin{proof}
This is corollary 4.7 of \cite{VisArs}. 
\end{proof}

The quotient of $X$ by the $\mathrm{G}_m$ action $(x_1,\ldots,x_9,t) \rightarrow (tx_1,\ldots,tx_9)$, which we will denote by $Z$, is naturally an open subset of the $\mathrm{PGL}_2 \times \mathrm{G}_m$-scheme  $P(\mathbb{A}^{9})$, namely the complement of the discriminant locus.

 For the sake of brevity we define $\mathrm{G} := \mathrm{PGL}_2 \times \mathrm{G}_m$. We will first construct the invariants of the quotient stack $\left[ Z/\mathrm{G} \right]$, then use the principal $\mathrm{G}_m$-bundle $\left[ X/\mathrm{G} \right] \rightarrow \left[ Z/\mathrm{G} \right]$ to compute the invariants of $\mathscr{H}_3$.

\vspace{0.5cm}

Let $F$ be the dual of the standard representation of $\mathrm{GL}_2$. We can see $F$ as the space of all binary forms $\phi=\phi(x_0,x_1)$ of degree $1$. It has the natural action of $\mathrm{GL}_2$ defined by $A(\phi)(x)=\phi(A^{-1}(x))$. We denote by $E_n$ the $n$-th symmetric power $\operatorname{Sym}^n(F)$. We can see $E_n$ as the space of all binary forms of degree $n$, and the action of $\mathrm{GL}_2$ induced by the action on $F$ is again $A(\phi)(x)=\phi(A^{-1}(x))$. If $n$ is even we can consider the additional action of $\mathrm{PGL}_2$ given by $\left[ A \right](\phi)(x)=\operatorname{Det}(A)^{n/2}f(A^{-1}(x))$.

We denote $\tilde{\Delta}_{r,n}$ the closed subspace of $E_n$ composed of forms $\phi$ such that there exists a form $f$ of degree $r$ whose square divides $\phi$. With this notation the scheme $X$ in theorem (\ref{H3_quot}) is equal to $E_{8} \! \smallsetminus \! \tilde{\Delta}_{1,8}$.

 We denote $\Delta_{r,n}$ the closed locus of the projectivization $P(E_n)$ consisting of forms $\phi$ such that there exists a form $f$ of degree $r$ whose square divides $\phi$. With this notation we have $Z = P(E_{8}) \! \smallsetminus \! \Delta_{1,8}$.

Thanks to the localization exact sequence on Chow groups with coefficients, understanding the cohomological invariants of $\left[ P(E_8) \! \smallsetminus \! \Delta_{1,8}/\mathrm{G} \right]$ can be reduced to understanding the invariants of $\left[ P(E_8) /\mathrm{G} \right]$, which are understood due to the projective bundle formula, the top Chow group with coefficients $A^0_\mathrm{G}(\Delta_{1,8})$ (which is not equal to the cohomological invariants of $\left[ \Delta_{1,8}/\mathrm{G} \right]$, as $\Delta_{1,8}$ is not smooth) and the pushforward map $A^0_\mathrm{G}(\Delta_{1,8})\rightarrow A^1_\mathrm{G}(P(E_8))$. The computation of $A^0_\mathrm{G}(\Delta_{1,8})$ will be based on the following result.

\begin{proposition}\label{univ}
The following results hold:
\begin{enumerate}

\item The pushforward of a (equivariant) universal homeomorphism induces an isomorphism on (equivariant) Chow groups with coefficients in $\operatorname{H}^{\mbox{\tiny{$\bullet$}}}$.

\item Let $\pi_{r,n}: P(E_{n-2r})\times P(E_{r})\rightarrow \Delta_{r,n}$ be the map induced by $(f,g)\rightarrow fg^2$. The equivariant morphism $\pi_{r,n}$ restricts to a universal homeomorphism on $\Delta_{r,n} \! \smallsetminus \! \Delta_{r+1,n}$.

\end{enumerate}
\end{proposition}
\begin{proof}

This was proven by the author in \cite[3.3,3.4]{Pir2}

\end{proof}

Lastly, in the next section we will mostly be able to ignore the action of $\mathrm{G}_m$ on $Z$ thanks to the following proposition. Note that $\mathrm{G}_m$ acts trivially on $Z$.

\begin{proposition}\label{Gm}
Let $T$ be a scheme with an action of $\mathrm{PGL}_2$ on it, and let $\mathrm{G_m}$ act on it trivially. Then the pullback through the map $\left[ T / \mathrm{PGL}_2 \right] \rightarrow \left[ T / \mathrm{PGL}_2 \times \mathrm{G}_m \right]$ induces an isomorphism on cohomological invariants. Moreover, we have $$A^{\mbox{\tiny{$\bullet$}}}_{\mathrm{PGL}_2 \times \mathrm{G}_m}(T)=A^{\mbox{\tiny{$\bullet$}}}_{\mathrm{PGL}_2 }(T)[s]$$ where $s$ is an element in codimension $1$ and degree zero.
\end{proposition}
\begin{proof}
Consider a representation $V$ of $\mathrm{PGL}_2$ such that $\mathrm{PGL}_2$ acts freely on an opens subset $U$ whose complement has codimension two or more. Given $n \geq 2$, let $\mathrm{G}_m$ act on $\mathbb{A}^n$ by multiplication. Then $\mathrm{PGL}_2 \times \mathrm{G}_m$ acts freely of $U \times (\mathbb{A}^n \! \smallsetminus \! \lbrace 0 \rbrace)$. As $\mathrm{G}_m$ acts trivially on $T$ we can see that $$(T \times U \times (\mathbb{A}^2 \! \smallsetminus \! \lbrace 0 \rbrace ) )/(\mathrm{PGL}_n \times \mathrm{G}_m) \simeq ((T \times U)/\mathrm{PGL}_2) \times P^{n-1}$$and pulling back $U \times ( \mathbb{A}^n \! \smallsetminus \! \lbrace 0 \rbrace )$ through $$\left[ T / \mathrm{PGL}_2 \right] \rightarrow \left[ T / \mathrm{PGL}_2 \times \mathrm{G}_m \right]$$we obtain the map $$((T \times U)/\mathrm{PGL}_2) \times (\mathbb{A}^n \! \smallsetminus \! \lbrace 0 \rbrace) \rightarrow (T \times U)/\mathrm{PGL}_2 \times P^{n-1}$$which induces an isomorphism on $A^{0}$ by the projective bundle formula, so by proposition (\ref{inv-chow}) it induces an isomorphism on cohomological invariants .

Finally, taking $n$ to infinity we get the required isomorphism on equivariant Chow groups with coefficients.
\end{proof}

\section{The invariants of $\mathscr{H}_3$}

In this section we will prove the main theorems of the paper. Thanks to proposition (\ref{Gm}) we will mostly be working with $\mathrm{PGL}_2$-equivariant Chow groups with coefficients. From now on we will shorten $P(E_n)$ to $P^n$.

 There are various differences from the case of even genus considered in \cite{Pir2}; the algebraic group $\mathrm{PGL}_2$ is not special, meaning that a $\mathrm{PGL}_2$-torsor is not in general Zariski-locally trivial. Consequently given a $\mathrm{PGL}_2$-scheme $X$ the map $X \rightarrow \left[ X/\mathrm{PGL}_2 \right]$ will not in general be a smooth-Nisnevich covering (definition 3.2 in \cite{Pir2}), and more importantly the $\mathrm{PGL}_2$-equivariant Chow groups with coefficients of $X$ will have multiple elements in positive degree coming from the projection $\left[ X/\mathrm{PGL}_2 \right] \rightarrow \textrm{BPGL}_2$ when $p=2$.

\begin{proposition}\label{PGL_2}
Let $p$ be equal to $2$, and $M=\operatorname{H}^{\mbox{\tiny{$\bullet$}}}$. Then $A^{\mbox{\tiny{$\bullet$}}}_{\mathrm{PGL}_2}(\operatorname{Spec}(k_0))$ is freely generated as a module over $\operatorname{CH}^{\mbox{\tiny{$\bullet$}}}_{\mathrm{PGL}_2}(\operatorname{Spec}(k_0))\otimes \operatorname{H}^{\mbox{\tiny{$\bullet$}}}(k_0)$ by the cohomological invariant $w_2$ and an element $\tau_{1,1}$ in degree and codimension $(1,1)$.

If $p \neq 2$, then $A^{\mbox{\tiny{$\bullet$}}}_{\mathrm{PGL}_2}(\operatorname{Spec}(k_0))$ is equal to $\operatorname{CH}^{\mbox{\tiny{$\bullet$}}}_{\mathrm{PGL}_2}(\operatorname{Spec}(k_0))\otimes \operatorname{H}^{\mbox{\tiny{$\bullet$}}}(k_0)$.

\end{proposition}
\begin{proof}
As $\mathrm{PGL}_2$ is isomorphic to $\mathrm{SO}_3$, we can just apply corollary (\ref{SO3}).
\end{proof}

The final difference is that the action of $\mathrm{PGL}_2$ on $P^1$ does not come from a linear action on the space of degree one forms. This is true for our $\mathrm{PGL}_2$ action on $P^n$ whenever $n$ is odd. The following proposition describes the ring $A^{\mbox{\tiny{$\bullet$}}}_{\mathrm{PGL}_2}(P^1)$.

\begin{proposition}\label{P^1}
Denote by $t$ the first Chern class of $\mathcal{O}_{P^{1}}(-1)$. Then $A^{\mbox{\tiny{$\bullet$}}}_{\mathrm{PGL}_2}(P^1)$ is isomorphic to $\operatorname{H}^{\mbox{\tiny{$\bullet$}}}\left[ t \right]$ and the image of $c_2 \in A^{2}_{\mathrm{PGL}_2}(\operatorname{Spec}(k_0))$ is $-t^2$.

If $p=2$, then the kernel of the map $\pi^*: A^{\mbox{\tiny{$\bullet$}}}_{\mathrm{PGL}_2}(\operatorname{Spec}(k_0))\rightarrow A^{\mbox{\tiny{$\bullet$}}}_{\mathrm{PGL}_2}(P^1)$ is generated by $w_2, c_3,\tau_{1,1}$. 
\end{proposition}
\begin{proof}
This can be proven exactly as in \cite[5.1]{FulgViv}; the group acts transitively on $P^1$ with stabilizer a group $\mathrm{H} \cong \mathrm{G}_m \ltimes \mathrm{G}_a$. This shows that $A^{^{\mbox{\tiny{$\bullet$}}}}_{\mathrm{PGL}_2}(P^1)$ must be isomorphic to $A_{\mathrm{H}}(\operatorname{Spec}(k_0))$. By lemma (\ref{eq2}) we see that $A_{\mathrm{H}}(\operatorname{Spec}(k_0))\cong A_{\mathrm{G}_m}(\operatorname{Spec}(k_0))=\operatorname{H}^{\mbox{\tiny{$\bullet$}}}\left[t\right]$. Then the computation follows from the one on equivariant Chow rings in \cite[5.1]{FulgViv}.
\end{proof}

We draw an outline of the main proof before getting into it, as it will require several steps.

 We begin by computing the cohomological invariants of $\left[ P^{n} \! \smallsetminus \! \Delta_{1,n}/ \mathrm{PGL}_2 \right]$, for $n \leq 8$ in the case of $p=2$ and for all $n$ in the case of $P \neq 2$. To do so we use the exact sequence $$0 \rightarrow A^{0}_{\mathrm{PGL}_2}(P^{n}) \rightarrow A^{0}_{\mathrm{PGL}_2}(P^{n} \! \smallsetminus \! \Delta_{1,n}) \rightarrow A^{0}_{\mathrm{PGL}_2}(\Delta_{1,n}) \rightarrow A^{1}_{\mathrm{PGL}_2}(P^{n}) .$$

After computing these invariants, we automatically get the invariants of $\left[ P^{n} \! \smallsetminus \! \Delta_{1,n}/ \mathrm{PGL}_2 \times \mathrm{G}_m \right]$ thanks to lemma (\ref{Gm}), and finally we are left to deal with the $\mathrm{G}_m$-torsor $\mathscr{H}_3 \rightarrow \left[ P^{n} \! \smallsetminus \! \Delta_{1,n}/ \mathrm{PGL}_2 \times \mathrm{G}_m \right]$. The steps are as follows:

\begin{enumerate}
\item In lemmas $17-18$ and corollary $19$ we establish that for $p=2$ we have isomorphism $$A^{0}_{\mathrm{PGL}_2}(\Delta_{1,n}) \simeq A^{0}_{\mathrm{PGL}_2}(\Delta_{1,n}\! \smallsetminus \! \Delta_{2,n}) \simeq A^{0}_{\mathrm{PGL}_2}((P^{n-2} \! \smallsetminus \! \Delta_{1,n-2}) \times P^1)$$ and moreover that $A^{0}_{\mathrm{PGL}_2}((P^{n-2} \! \smallsetminus \! \Delta_{1,n}) \times P^1)$ can be obtained as a quotient of $A^{0}_{\mathrm{PGL}_2}(P^{n-2} \! \smallsetminus \! \Delta_{1,n})$, setting up an inductive computation.
\item In lemma $20$ we prove that for $p \neq 2$ the group $A^{0}_{\mathrm{PGL}_2}(\Delta_{1,n})$ is a trivial $\operatorname{H}^{\mbox{\tiny{$\bullet$}}}(k_0)$-module.
\item In lemma $21$, proposition $22$ and corollary $23$ we show that when we have $p=2, n \leq 8$ the pushforward $A^{0}_{\mathrm{PGL}_2}(\Delta_{1,n}) \rightarrow A^{1}_{\mathrm{PGL}_2}(P^{n})$ is zero. To do so we will construct an element $g_n \in A^{\mbox{\tiny{$\bullet$}}}_{\mathrm{PGL}_2}(P^{n})$ which annihilates the image of $A^{0}_{\mathrm{PGL}_2}(\Delta_{1,n})$ but at the same does not annihilate any non-zero element of $A^{1}_{\mathrm{PGL}_2}(P^{n})$
\item In corollary $23$ we use the localization exact sequence for $\Delta_{1,n} \rightarrow P^n$, now reduced to a short exact sequence, to compute the cohomological invariants of $\left[ P^{n} \! \smallsetminus \! \Delta_{1,n}/ \mathrm{PGL}_2 \right]$ for $n \leq 8$ when $p=2$ and for all $n$ when $p \neq 2$.
\item In lemma $24$ and theorem $25$ we prove the main result. What is left to do is understanding whether the $\mathrm{G}_m$-torsor $\mathscr{H}_3 \rightarrow \left[ P^{n} \! \smallsetminus \! \Delta_{1,n}/ \mathrm{PGL}_2 \times \mathrm{G}_m \right]$ generates any new invariant, which boils down to understanding the kernel of the first Chern class $c_1(\mathcal{E})$, where $\mathcal{E}$ is the line bundle associated to the $\mathrm{G}_m$-torsor.
\end{enumerate}

We first tackle the case the case $p=2$, which will prove to be a bit delicate. The next lemmas will show that several different statements regarding various maps imply each other. For an even positive integer $n$, consider the following statements:

\begin{description}
\item[$\mathrm{S}_1(n)$:] the pullback $$A^0_{\mathrm{PGL}_2}(P^{n}\! \smallsetminus \! \Delta_{1,n}) \xrightarrow{\pi^*} A^0_{\mathrm{PGL}_2}((P^{n}\! \smallsetminus \! \Delta_{1,n})\times P^1)$$ is surjective and the kernel of $\pi^*$ is generated by $w_2$, the second Stiefel-Whitney class coming from $A^0_{\mathrm{PGL}_2}(\operatorname{Spec}(k_0))=\operatorname{Inv}^{\mbox{\tiny{$\bullet$}}}(\mathrm{BPGL}_2)$.

\item[$\mathrm{S}_2(n)$:] the pullback $$A^0_{\mathrm{PGL}_2}(\Delta_{1,n})\rightarrow A^0_{\mathrm{PGL}_2}(\Delta_{1,n}\! \smallsetminus \! \Delta_{2,n})$$ is an isomorphism.

\item[$\mathrm{S}_3(n)$:] the pullback $$A^0_{\mathrm{PGL}_2}(\operatorname{Spec}(k_0))\rightarrow A^0_{\mathrm{PGL}_2}(\Delta_{2,n})$$ is an isomorphism.
\end{description}
Note that proposition (\ref{P^1}) implies $\mathrm{S}_1(0)$. We have the following implications between the statements above:

\begin{lemma}\label{recurs1}
Let $p$ be equal to $2$.  If $\mathrm{S}_1(n-i)$ holds for all $i \geq 2$ then $\mathrm{S}_2(n)$ and $\mathrm{S}_3(n)$ hold.
\end{lemma}
\begin{proof}
To prove the first point, we want to repeat the proof of \cite[4.4]{Pir2} basically word for word. There is only one additional statement that we have to prove when working with $\textrm{PGL}_2$ instead of $\textrm{GL}_2$, the fact that that given a $\mathrm{PGL}_2$ scheme $X$ the pullback through $X \times P^1 \times P^1 \rightarrow X \times P^1$ is an isomorphism on $A^{0}_{\mathrm{PGL}_2}(-)$.

The group $\mathrm{PGL}_2$ acts transitively on $P^1$, with stabilizer $\mathrm{H}\simeq \mathrm{G}_a\rtimes \mathrm{G}_m$. Then we have $\left[P^1/\mathrm{PGL}_2\right]\simeq \mathrm{BH}$, so $\left[P^1\times P^1/\mathrm{PGL}_2\right] \simeq \left[P^1/\mathrm{H} \right]$, and moreover the action of $\mathrm{H}$ can be lifted to a linear action on the vector space $F=E_1$. Then shows that given a $\mathrm{PGL}_2$-equivariant space $X$, we have $$\left[X\times P^1/\mathrm{PGL}_2\right]=\left[X/\mathrm{H}\right], \left[X\times P^1 \times P^1/\mathrm{PGL}_2\right]=\left[X \times P^1/\mathrm{H}\right]$$ and thus the pullback through the $\mathrm{PGL}_2$ equivariant projection $X \times P^1  \times P^1 \rightarrow X \times P^1$ is the same as the $\mathrm{H}$-equivariant pullback through $X \times P^1 \rightarrow X$ which is an isomorphism in codimension zero by the projective bundle formula.

Using this we have all the tools to repeat the diagram chase in \cite[4.4]{Pir2} step by step and prove the first point. For the sake of self containment we will repeat the proof. First, note that the case $n=2$ is trivial. Let $r \in \lbrace 1,2 \rbrace$.  As $A^{0}_{\mathrm{PGL}_2}(\Delta_{r,n})$ is isomorphic to $A^{0}_{\mathrm{PGL}_2}(\Delta_{r,n} \! \smallsetminus \! \Delta_{r+2,n})$ (because $\Delta_{r+2,n}$ has codimension two in $\Delta_{r,n}$) we can compute it using the following exact sequence:

 $$ 0 \rightarrow A^{0}_{\mathrm{PGL}_2}(\Delta_{r,n} \! \smallsetminus \! \Delta_{r+2,n}) \rightarrow A^{0}_{\mathrm{PGL}_2}(\Delta_{r,n} \! \smallsetminus \! \Delta_{r+1,n}) \xrightarrow{\partial} A^{0}_{\mathrm{PGL}_2}(\Delta_{r+1,n} \! \smallsetminus \! \Delta_{r+2,n}).$$
 
 When $r=2$, we want to prove that the kernel of $\partial$ is equal to the image of $A^{0}_{\mathrm{PGL}_2}(\operatorname{Spec}(k_0))$. This will then imply that $A^{0}_{\mathrm{PGL}_2}(\Delta_{r,n} \! \smallsetminus \! \Delta_{r+2,n})$ must be equal to $A^{0}_{\mathrm{PGL}_2}(\operatorname{Spec}(k_0))$. When $r=1$, we want to prove that $\partial$ is zero, so that the second arrow will be an isomorphism.
 
  The map $(P^{n-2r} \! \smallsetminus \! \Delta_{2,2r}) \times P^{r} \xrightarrow{\pi} \Delta_{r,n} \! \smallsetminus \! \Delta_{r+2,n}$ yields the following commutative diagram with exact columns:

\begin{center}
$\xymatrixcolsep{3pc}
\xymatrix{ A^{0}_{\mathrm{PGL}_2}((P^{n-2r} {\! \smallsetminus \!} \Delta_{2,n-2r}) \times P^r ) \ar@{->}[r]^{\pi_{*}} \ar@{->}[d] & A^{0}_{\mathrm{PGL}_2}(\Delta_{r,n} \! \smallsetminus \! \Delta_{r+2,n} )   \ar@{->>}[d] \\
A^{0}_{\mathrm{PGL}_2}((P^{n-2r} \! \smallsetminus \! \Delta_{1,n-2r}) \times P^r) \ar@{^{(}->>}[r]^{\pi_{*}} \ar@{->}[d]^{\partial_1} &  A^{0}_{\mathrm{PGL}_2}(\Delta_{r,n}\! \smallsetminus \! \Delta_{r+1,n}) \ar@{->}[d]^{\partial}\\
 A^{0}_{\mathrm{PGL}_2}((\Delta_{1,n-2r} \! \smallsetminus \! \Delta_{2,n-2r})\times P^r) \ar@{->}[r]^{\pi_{*}} & A^{0}_{\mathrm{PGL}_{2}}(\Delta_{r+1,n}\! \smallsetminus \! \Delta_{r+2,n})  } $
\end{center}

 The second horizontal map is an isomorphism because $\pi_*$ is a universal homeomorphism when restricted to $\Delta_{r,n} \! \smallsetminus \! \Delta_{r+1,n}$.
 
  The kernel of $\partial_{1}$ is the image of $A^{0}_{\mathrm{PGL}_2}(\operatorname{Spec}(k_0))$, as $\Delta_{2,n-2r} \times P^r$ has codimension $2$. 
 
  We claim that when $r=2$ the third horizontal map is an isomorphism, implying that the kernel of $\partial$ must also be the image of $A^{0}_{\mathrm{PGL}_2}(\operatorname{Spec}(k_0))$, and when $r=1$ the third horizontal map is zero, so that $\partial$ must be zero too.
 
 Let $\psi$ be the map $$(P^{n-2r-2} \! \smallsetminus \! \Delta_{1,n-2r-2})\times P^{r} \times P^1 \xrightarrow{\psi} (P^{n-2r-2} \! \smallsetminus \! \Delta_{1,n-2r-2})\times P^{r+1}$$ sending $(f,g,h)$ to $(f,gh)$. We have a commutative diagram:
 
\begin{center}
$\xymatrixcolsep{5pc}
\xymatrix{ (P^{n-2r-2} \! \smallsetminus \! \Delta_{1,n-2r-2}) \times P^1 \times P^{r}  \ar@{->}[r]^{\pi_1} \ar@{->}[d]^{\psi} & (\Delta_{1,n -2r} \! \smallsetminus \! \Delta_{2,n-2r}) \times P^r \ar@{->}[d]^{\pi}\\ 
(P^{n-2r-2} \! \smallsetminus \! \Delta_{1,n-2r-2})\times P^{r+1}  \ar@{->}[r]^{\pi_2} & \Delta_{r +1,n} \! \smallsetminus \! \Delta_{r +2,n}}$
\end{center}

Where $\pi_1$ and $\pi_2$ are defined respectively by $(f,g,h) \rightarrow (fg^2,h)$ and $(f,g) \rightarrow (fg^2)$. The maps $\pi_1$ and $\pi_2$ are universal homeomorphisms, so the pushforward maps $(\pi_1)_*,(\pi_2)_*$ are isomorphisms. Then if we prove that $\psi_*$ is an isomorphism $\pi_*$ will be an isomorphism too, and if $\psi_*$ is zero then $\pi_*$ will be zero too. Consider this last diagram:
 
\begin{center}
$\xymatrixcolsep{5pc}
\xymatrix{ (P^{n-2r-2} \! \smallsetminus \! \Delta_{1,n-2r-2})\times P^{r} \times P^1 \ar@{->}[dr]^{p_1} \ar@{->}[d]^{\psi}\\ 
(P^{n-2r-2} \! \smallsetminus \! \Delta_{1,n-2r-2})\times P^{r+1}  \ar@{->}[r]^{p_2} & P^{n-2r-2} \! \smallsetminus \! \Delta_{1,n-2r-2}}$
\end{center}
The pullbacks along $p_1$ and $p_2$ are both surjective, implying that the pullback of $\psi$ is surjective. We have $\psi_* ( \psi^* \alpha) = \operatorname{deg}(\psi) \alpha$ by the projection formula. Then as the degree of $\psi$ is $r + 1$, $\psi_*$ is an isomorphism if $r=2$ and zero if $r=1$.

\end{proof}

\begin{lemma}\label{recurs2}
Let $p$ be equal to $2$. Suppose that $\mathrm{S}_2(n)$ holds and that the pushforward $$A^0_{\mathrm{PGL}_2}(\Delta_{1,n}) \rightarrow A^1_{\mathrm{PGL}_2}(P^{n})$$ is zero. Then $\mathrm{S}_1(n)$ holds.
\end{lemma}
\begin{proof}
Consider the following commutative diagram with exact columns:

\begin{center}
$\xymatrixcolsep{3pc}
\xymatrix{ 0 \ar@{->}[d] & 0 \ar@{->}[d]\\
A^{0}_{\textrm{PGL}_2}(P^n) \ar@{->>}[r] \ar@{->}[d] & A^{0}_{\textrm{PGL}_2}(P^n \times P^1 )   \ar@{->}[d] \\
A^{0}_{\textrm{PGL}_2}((P^{n} \! \smallsetminus \! \Delta_{1,n})) \ar@{->}[r] \ar@{->}[d] &  A^{0}_{\textrm{PGL}_2}(((P^{n} \! \smallsetminus \! \Delta_{1,n}))\times P^1) \ar@{->}[d]\\
 A^{0}_{\textrm{PGL}_2}((\Delta_{1,n})) \ar@{->}[r] \ar@{->}[d] & A^{0}_{\textrm{PGL}_{2}}(\Delta_{1,n}\times P^1) \\
 0 } $
\end{center}

We know that the left column is exact as the map $A^{0}_{\textrm{PGL}_2}(\Delta_{1,n}) \rightarrow A^1_{\mathrm{PGL}_2}(P^n)$ is zero by hypothesis. The fact that the topmost horizontal map is surjective can be seen exactly as for $P^1$. A simple diagram chase shows that if the last horizontal map is surjective, then the central horizontal map must be surjective too. To prove this we use a second commutative diagram with exact columns:

$$
\xymatrixcolsep{3pc}
\xymatrix{ 0 \ar@{->}[d] & 0 \ar@{->}[d]\\
A^{0}_{\textrm{PGL}_2}(\Delta_{1,n}) \ar@{->}[r] \ar@{->}[d] & A^{0}_{\textrm{PGL}_2}(\Delta_{1,n} \times P^1 )   \ar@{->}[d] \\
A^{0}_{\textrm{PGL}_2}((\Delta_{1,n} \! \smallsetminus \! \Delta_{2,n})) \ar@{->}[r] \ar@{->}[d] &  A^{0}_{\textrm{PGL}_2}(((\Delta_{1,n} \! \smallsetminus \! \Delta_{2,n}))\times P^1) \ar@{->}[d]\\
0 \ar@{->}[r] & A^{0}_{\textrm{PGL}_{2}}(\Delta_{2,n}\times P^1)  } 
$$

The left column is exact thanks to $\mathrm{S}_2(n)$. To conclude we only need to prove that the central horizontal map is surjective. But this is just the map $$A^0_{\mathrm{PGL}_2}((P^{n-2}\! \smallsetminus \! \Delta_{1,n-2})\times P^1)\rightarrow A^0_{\mathrm{PGL}_2}((P^{n-2}\! \smallsetminus \! \Delta_{1,n-2})\times P^1 \times P^1)$$which is an isomorphism. 

This also tells us that the map $$A^{0}_{\textrm{PGL}_2}(\Delta_{1,n}) \rightarrow A^{0}_{\textrm{PGL}_2}(\Delta_{1,n} \times P^1 ) $$ is an isomorphism, so the elements in the kernel of $$A^0_{\mathrm{PGL}_2}(P^{n} \! \smallsetminus \! \Delta_{1,n})\rightarrow A^0_{\mathrm{PGL}_2}((P^{n} \! \smallsetminus \! \Delta_{1,n})\times P^1)$$must come from $A^0_{\mathrm{PGL}_2}(P^n)$, which completes our description.
\end{proof}

The lemmas almost provides an inductive step, as its conclusions provide all of the hypotheses for the next case except for the requirement that the pushforwards $A^0_{\mathrm{PGL}_2}(\Delta_{1,n}) \rightarrow A^1_{\mathrm{PGL}_2}(P^{n})$ are zero.

\begin{corollary}\label{recurs3}
Suppose that for all $j \leq n$ we know that that the pushforward $A^0_{\mathrm{PGL}_2}(\Delta_{1,j}) \rightarrow A^1_{\mathrm{PGL}_2}(P^{j})$ is zero. Then for $j \leq n$ the conditions $\mathrm{S}_{1}(j),\mathrm{S}_2(j)$ and $\mathrm{S}_3(j)$ hold.
\end{corollary}

\begin{proof}
Given the hypothesis and the trivial cases $j=0,j=2$ lemmas (\ref{recurs1}, \ref{recurs2}) inductively prove all three properties for all $j \leq n$.
\end{proof}

The statement needed for $p \neq 2$ is more straightforward, although it relies on the same argument.

\begin{proposition}\label{even1}
Suppose $p$ is different from $2$. Then $A^{0}_{\mathrm{PGL}_2}(\Delta_{1,n})$ is trivial.
\end{proposition}
\begin{proof}
 We want to use the same reasoning as in the lemma (\ref{recurs1}). Then at the last point we will obtain that $\psi_*$ is an isomorphism if $r+1$ does not divide $p$, which is what happens for $r=1$, proving our claim. All of the diagram chases in the previous lemma work for $p \neq 2$, so we only have to show that the map $$(P^{n-4} \! \smallsetminus \! \Delta_{1,n-4})\times P^{1} \times P^1 \rightarrow (P^{n-4} \! \smallsetminus \! \Delta_{1,n-4})$$ induces a surjective pullback on $A^0_{PGL_2}(-)$. To do so, note the following. We have $$A^0_{\mathrm{PGL}_2}((P^{n-4} \! \smallsetminus \! \Delta_{1,n-4})\times P^{1} \times P^1) \simeq A^{0}_{\mathrm{H}}((P^{n-4} \! \smallsetminus \! \Delta_{1,n-2r-2})\times P^{1})$$
 where $\mathrm{H}$ is the stabilizer of a point in $P^1$ as above. As $\mathrm{H}$ is a special group the pullback $$A^{0}_{\mathrm{H}}((P^{n-4} \! \smallsetminus \! \Delta_{1,n-4})\times P^{1}) \rightarrow  A^{0}((P^{n-4} \! \smallsetminus \! \Delta_{1,n-4})\times P^{1})$$ has to be injective. Now one can use the same techniques as in \cite[4.4]{Pir2}, or equivalently as in the previous lemma to easily show that when $p \neq 2$ the non-equivariant group $A^{0}( \Delta_{1,n-4}\times P^{1})$ is trivial, and thus $A^{0}((P^{n-4} \! \smallsetminus \! \Delta_{1,n-4})\times P^{1})$ is either trivial or generated by $1$ and an element in degree one corresponding to an equation for $\Delta_{1,n-4}$ if the class of $\Delta_{1,n-4}$ is equal to zero in $A^1(P^{n-4}\times P^1)$. In the latter case, consider the following commutative diagram induced by the pullback from equivariant to non-equivariant Chow groups with coefficients
 
 \begin{center}
$\xymatrixcolsep{3pc}
\xymatrix{  A^{0}_{\mathrm{PGL}_2}(P^{n-4}) \ar@{->}[r] \ar@{->}[d] & A^{0}(P^{n-4})   \ar@{->}[d] \\
A^{0}_{\mathrm{PGL}_2}(P^{n-4} \! \smallsetminus \! \Delta_{1,n-4}) \ar@{->}[r] \ar@{->}[d]^{\partial} &  A^{0}(P^{n-4} \! \smallsetminus \! \Delta_{1,n-4}) \ar@{->}[d]^{\partial}\\
 A^{0}_{\mathrm{PGL}_2}(\Delta_{1,n-4} ) \ar@{->}[r] \ar@{->}[d] & A^{0}(\Delta_{1,n-4})  \ar@{->}[d]\\
  A^{1}_{\mathrm{PGL}_2}(P^{n-4}) \ar@{->}[r]  & A^{1}(P^{n-4})} $
\end{center}

The top and bottom horizontal maps are isomorphisms, and one can see using the fact that both groups on top are trivial an both groups on the bottom are generated as $\operatorname{H}^{\mbox{\tiny{$\bullet$}}}$-module by the first Chern class of $\mathcal{O}_{P^{n-4}}(-1)$. Moreover $A^{0}(P^{n-4} \! \smallsetminus \! \Delta_{1,n-4})$ is generated as a $\operatorname{H}^{\mbox{\tiny{$\bullet$}}}$-module by $1$ and an element $\alpha$ such that $\partial(\alpha)=1 \in A^{0}_{\mathrm{PGL}_2}(\Delta_{1,n-4} )$. 

The third horizontal map maps $1 \in A^{0}_{\mathrm{PGL}_2}(\Delta_{1,n-4} )$ to $1 \in A^{0}(\Delta_{1,n-4} )$, which shows that $1$ maps to zero in the equivariant group $ A^{1}_{\mathrm{PGL}_2}((P^{n-4})$ if and only if it maps to zero in $ A^{1}((P^{n-4})$. Then there must be an element $$\alpha' \in  A^{0}_{\mathrm{PGL}_2}((P^{n-4}\! \smallsetminus \! \Delta_{1,n-4})$$ which maps to $\alpha \in  A^{0}((P^{n-4}\! \smallsetminus \! \Delta_{1,n-4})$, showing that the pullback $$A^{0}_{\mathrm{PGL}_2}(P^{n-4} \! \smallsetminus \! \Delta_{1,n-4}) \rightarrow A^0((P^{n-4} \! \smallsetminus \! \Delta_{1,n-4})\times P^1)$$ is surjective. This implies surjectivity for $$A^0_{\mathrm{PGL}_2}(P^{n-4} \! \smallsetminus \! \Delta_{1,n-4}) \rightarrow A^0_{\mathrm{PGL}_2}((P^{n-4} \! \smallsetminus \! \Delta_{1,n-4})\times P^{1} \times P^1),$$ as claimed.
\end{proof}

In the rest of the section we will slightly abuse notation and always denote by $t$ the (equivariant) class $c_1(\mathcal{O}_{P^n}(-1))$, independently of $n$. When in presence of a product $P^n \times P^m$ we will always denote by $t$ the one coming from the first component. 

Note that the pullback of $\mathcal{O}_{P^n}(-1)$ through the maps $i\circ\pi_{r,n}:P^{n-2r}\times P^r \rightarrow P^{n}$ is equal to $\operatorname{p_1}^*\mathcal{O}_{P^{n-2r}}(-1) \otimes \operatorname{p_2}^*\mathcal{O}_{P^{r}}(-1)^2$, so with the notation above when $p=2$ we have $(i\circ\pi_{r,n})^*t=t $.

\vspace{0.5cm}

Let $n$ be an even positive integer. By  the projective bundle formula we have $A^{\mbox{\tiny{$\bullet$}}}_{\mathrm{PGL}_2}(P^n)=A^{\mbox{\tiny{$\bullet$}}}_{\mathrm{PGL}_2}(\operatorname{Spec}(k_0))\left[ t \right]/(f_n)$ for some polynomial $f_n$ that is monic of degree $n+1$ in $t$. By \cite[6.1]{FulgViv} the $f_n$ are the following elements of $A^{\mbox{\tiny{$\bullet$}}}_{\mathrm{PGL}_2}(P^n)$:

$$f_n = \begin{cases} t^{n+4/4}(t^3 + c_2t + c_3)^{n/4}, & \mbox{if } n\mbox{ is divisible by 4} \\ t^{n-2/4}(t^3 + c_2t + c_3)^{n+2/4}, & \mbox{if } n\mbox{ is not.} \end{cases}$$

\begin{lemma}\label{c3odd}
Suppose that $p=2$. Then the class of $c_3$ is zero in $A_{\mathrm{PGL}_2}^{\mbox{\tiny{$\bullet$}}}(P^n)$ if and only if $n$ is odd.
\end{lemma}
\begin{proof}
If $n$ is even then $P^n$ is the projectivization of a representation of $\mathrm{PGL}_2$ and the projective bundle formula allows us to conclude immediately. If $n$ is odd we just have apply the projection formula to the equivariant map $P^1 \times P^{i-1} \rightarrow P^{i}$ and use the result for $n=1$, which is proven in proposition (\ref{P^1}).
\end{proof}

We can use this to construct an element in the annihilator of the image of the pushforward $i_*(A^0_{\mathrm{PGL}_2}(\Delta_{1,n}))$.

\begin{proposition}
 Let $n$ be an even positive integer, and let $\alpha$ be an element of $A^0_{\mathrm{PGL}_2}(\Delta_{1,n})$. Then:
\begin{itemize}
\item If $n$ is divisible by $4$, the image of $\alpha$ in $A^{\mbox{\tiny{$\bullet$}}}_{\mathrm{PGL}_2}(P^n)$ is annihilated by $c_3^{n/4} f_{n-4}\ldots f_4 t$.
\item If $n$ is not divisible by $4$, the image of $\alpha$ in $A^{\mbox{\tiny{$\bullet$}}}_{\mathrm{PGL}_2}(P^i)$ is annihilated by $c_3^{n+2/4} f_{n-4}\ldots f_2$.
\end{itemize}
\end{proposition}
\begin{proof}
Let $i: \Delta_{1,n} \rightarrow P^n$ be the inclusion. We will also denote by $i$ all of its restrictions. Consider the sequence of maps $$P^{n-2}\! \smallsetminus \! \Delta_{n-2,1} \times P^1 \rightarrow \Delta_{1,n}\! \smallsetminus \! \Delta_{2,n} \xrightarrow{i} P^n \! \smallsetminus \! \Delta_{2,n}.$$  The pullback of $c_3$ to $A^{\mbox{\tiny$\bullet$}}_{\mathrm{PGL}_2}((P^{n-2}\! \smallsetminus \! \Delta_{n-2,1}) \times P^1)$ is zero by lemma (\ref{c3odd}) and $\Delta_{1,n}\! \smallsetminus \! \Delta_{2,n}$ is universally homeomorphic to $P^{n-2}\! \smallsetminus \! \Delta_{n-2,1} \times P^1$. Then by the compatibility of Chern classes with pushforwards that the pullback of $c_3$ through $i$ must be zero. This shows that $c_3 i_* \alpha =0$. As we already know that $c_3 i_* \alpha$ belongs to $A^{\mbox{\tiny{$\bullet$}}}_{\mathrm{PGL}_2}(P^n)$ it must belong to
 $$\operatorname{Ker}(A^{\mbox{\tiny{$\bullet$}}}_{\mathrm{PGL}_2}(P^n) \rightarrow A^{\mbox{\tiny{$\bullet$}}}_{\mathrm{PGL}_2}(P^n \! \smallsetminus \! \Delta_{2,n}))= i_* A^{\mbox{\tiny{$\bullet$}}}_{\mathrm{PGL}_2}(\Delta_{2,n}).$$
 
  Let $\beta \in A^2(\Delta_{2,n})$ be a preimage of $c_3 i_* \alpha$, and consider the sequence of maps $$P^{n-4}\! \smallsetminus \! \Delta_{n-4,1} \times P^2 \rightarrow \Delta_{2,n}\! \smallsetminus \! \Delta_{3,n} \xrightarrow{i} P^n \! \smallsetminus \! \Delta_{3,n}.$$ Let $\beta'$ be the pullback of $\beta$ to $\Delta_{n,2} \! \smallsetminus \! \Delta_{n,3}$. We can see $\beta'$ as an element of $A^2_{\mathrm{PGL}_2}((P^{n-4} \! \smallsetminus \! \Delta_{1,n-4}) \times P^2)$. we know that in this ring the equation $f_{n-4}(t,c_2,c_3)=0$ holds and as we are working mod $2$ the pullback of $t \in A^{1}_{\mathrm{PGL}_2}(P^n)$ is equal to $t \in A^1_{\mathrm{PGL}_2}(P^{i-4}\times P^2)$. Then we have 
$$i^* f_{n-4}(t,c_2,c_3)  = f_{n-4}(t,c_2,c_3)=0 \in A^{\mbox{\tiny{$\bullet$}}}((P^{n-4}\! \smallsetminus \! \Delta_{1,n-4})\times P^{2})$$
implying that $f_{n-4}(t,c_2,c_3) i_* \beta'=0$ in $A^{\mbox{\tiny{$\bullet$}}}_{\mathrm{PGL}_2}(P^n \! \smallsetminus \! \Delta_{3,n})$. As before, this proves that $c_3 f_{n-4} i_* \alpha$ belongs to the image of $A^{\mbox{\tiny{$\bullet$}}}_{\mathrm{PGL}_2}(\Delta_{3,n})$.

We can clearly repeat this reasoning inductively to move from $\Delta_{r,i}$ to $\Delta_{r+1,n}$, multiplying by $c_3$ and applying lemma (\ref{c3odd}) if $r$ is odd, and multiplying by $f_{n-2r}$ is $r$ is even. The last thing to note is that when $r=n/2$ the process ends and we obtain $0$, either multiplying by $f_0=t$ if $n$ is divisible by $4$ or by $c_3$ otherwise.
\end{proof}

\begin{corollary}
Assume $p=2$. then the maps $i_* : A^0_{\mathrm{PGL}_2}(\Delta_{1,n}) \rightarrow A^1_{\mathrm{PGL}_2}(P^n)$ are zero for $n \leq 8$.
\end{corollary}
\begin{proof}
Let $\alpha$ be an element of $A^0_{\mathrm{PGL}_2}(\Delta_{1,n})$. Its pushforward $i_* \alpha$ must be of the form $t \beta + \tau_{1,1}\gamma$ for some $\beta \in A^0_{\mathrm{PGL}_2}(\operatorname{Spec}(k_0))$ and some $\gamma \in \operatorname{H}^{\mbox{\tiny{$\bullet$}}}(k_0)$. We know by the previous lemma that $g_n i_* \alpha=0$ for an appropriate polynomial $g_n$ in $t, c_2, c_3$. Now it suffices to note that for $g_n i_* \alpha$ can only be zero if both $g_n t \beta$ and $ g_n \tau_{1,1} \gamma$ are zero. The first requires that either $\alpha = 0$ or $f_n \mid g_n t$. The second can only happen if $\gamma = 0 $ or $f_n \mid g_n$. For $n \leq 8$ $f_n$ does not divide $g_n t$, so we can conclude that both $\beta$ and $\gamma$ must be zero.
\end{proof}

Note that the reasoning above does not work for any $n>8$. Higher genus cases will require a different idea.

\begin{corollary}
Let $p=2$. Then for all even $2 \leq n \leq 8$ the cohomological invariants $\operatorname{Inv}^{\mbox{\tiny{$\bullet$}}}(\left[ P^n\! \smallsetminus \!\Delta_{1,n} /\mathrm{PGL}_2 \right])$ are freely generated as a $\operatorname{H}^{\mbox{\tiny{$\bullet$}}}(k_0)$-module by $1$ and elements $ x_1, \ldots , x_{n/2}, w_2 $, where the degree of $x_i$ is $i$ and $w_2$ is the second Stiefel-Whitney class coming from the cohomological invariants of $\mathrm{PGL}_2$.

If $p\neq 2$, then the cohomological invariants of $\left[ P^n\! \smallsetminus \!\Delta_{1,n}  /\mathrm{PGL}_2 \right]$ are trivial unless $p$ divides $n-1$, in which case they are generated as a $\operatorname{H}^{\mbox{\tiny{$\bullet$}}}(k_0)$-module by $1$ and a single nonzero invariant of degree $1$.
\end{corollary}
\begin{proof}

Assume $p=2$. The previous lemma allows us to apply corollary (\ref{recurs3}) repeatedly, together with the exact sequence
$$ 0 \rightarrow A^0_{\mathrm{PGL}_2}(P^n) \rightarrow A^0_{\mathrm{PGL}_2}(P^n \! \smallsetminus \! \Delta_{1,n} ) \rightarrow A^{0}_{\mathrm{PGL}_2}(\Delta_{1,n}) \rightarrow 0.$$

We know these groups for $P^2$ and $\Delta_{1,2}$ (which is isomorphic to $P^1$). Starting with these we can use the exact sequence to compute the groups inductively (using $A^{0}_{\mathrm{PGL}_2}(\Delta_{1,n}) \simeq A^{0}_{\mathrm{PGL}_2}((P^{n-2} \! \smallsetminus \! \Delta_{1,n-2}) \times P^1)$). At the $n$-th step we get that $$A^0_{\mathrm{PGL}_2}(P^n \! \smallsetminus \! \Delta_{1,n} )\simeq A^0_{\mathrm{PGL}_2}(P^n)  \oplus A^{0}_{\mathrm{PGL}_2}(\Delta_{1,n})\left[1 \right]$$ where the $\left[ 1 \right]$ means we are shifting all degrees up by $1$; note that the $\operatorname{H}^{\mbox{\tiny{$\bullet$}}}(k_0)$-modules in the exact sequence are all free, so it splits each time.

The case $p \neq 2$ is easy: we need to check the next step of the exact sequence, that is, the pushforward map $A^{0}_{\mathrm{PGL}_2}(\Delta_{1,n})\xrightarrow{i_*} A^{1}_{\mathrm{PGL}_2}(P^n)$. As $A^{0}_{\mathrm{PGL}_2}(\Delta_{1,n})$ and $A^0_{\mathrm{PGL}_2}(P^n)$ are both generated by $1$ as $\operatorname{H}^{\mbox{\tiny{$\bullet$}}}(k_0)$-modules, we only need to look at the image of $1$ through the map $i_*$. The image of $1$ is the class of $\Delta_{1,n}$ which is a multiple of $t$ in $A^{1}_{\mathrm{PGL}_2}(P^n)=\operatorname{H}^{\mbox{\tiny{$\bullet$}}}(k_0)\cdot t$, divisible by $p$ if and only if $p$ divides $n-1$.
\end{proof}

Before we complete our computation, we need one last lemma. Recall that by lemma (\ref{Gm}) we have $$A^{\mbox{\tiny{$\bullet$}}}_{\mathrm{PGL}_2 \times \mathrm{G}_m}(P^{8} \! \smallsetminus \! \Delta_{1,8})=A^{\mbox{\tiny{$\bullet$}}}_{\mathrm{PGL}_2}(P^{8} \! \smallsetminus \! \Delta_{1,8})[s]$$ where $s$ is an element in codimension $1$ and degree $0$.

\begin{lemma}\label{c1Gm}
Let $n$ be an odd integer. Consider the $\mathrm{PGL}_2 \times \mathrm{G}_m$ equivariant $\mathrm{G}_m$-torsor $$\left[ \mathbb{A}^{2n+3} \! \smallsetminus \! \Delta /\mathrm{PGL}_2 \times \mathrm{G}_m \right] \rightarrow \left[ P^{2n+2} \! \smallsetminus \! \Delta_{1,2n+2} /\mathrm{PGL}_2 \times \mathrm{G}_m \right]$$ and let $\mathcal{E}_n$ be the $\mathrm{PGL}_2 \times \mathrm{G}_m$ equivariant line bundle obtained by completing it. Then the class of $c_1(\mathcal{E}_n)$ in $A^{1}_{\mathrm{PGL}_2 \times \mathrm{G}_m}(P^{2n+2} \! \smallsetminus \! \Delta_{1,2n+2})$ is equal to $t-2s$.
\end{lemma}
\begin{proof}
This is proven in \cite[eq. 3.2]{FulgViv}. Note that using the notation in \emph{loc.cit.} we have $d=n+1, r=2$.
\end{proof}

\begin{theorem}
Suppose that $p=2$ and $k_0$ is algebraically closed. Then the cohomological invariants of $\mathscr{H}_3$ are freely generated as an $\operatorname{H}^{\mbox{\tiny{$\bullet$}}}(k_0)$-module by $1$ and $x_1,x_2,w_2,x_3,x_4,x_5$, where the degree of $x_i$ is $i$ and $w_2$ is the second Stiefel-Whitney class coming from the cohomological invariants of $\mathrm{PGL}_2$.

In general, for $p=2$ the cohomological invariants of $\mathscr{H}_3$ fit in an exact sequence 
$$0\rightarrow M \rightarrow \operatorname{Inv}^{\mbox{\tiny{$\bullet$}}}(\mathscr{H}_3) \rightarrow K \rightarrow 0$$
 where $K$ is isomorphic to a submodule of $\operatorname{H}^{\mbox{\tiny{$\bullet$}}}(k_0)$, shifted up in degree by $5$, and $M$ is freely generated as a $\operatorname{H}^{\mbox{\tiny{$\bullet$}}}(k_0)$-module by $1$ and $x_1,x_2,w_2,x_3,x_4$, where the degree of $x_i$ is $i$ and $w_2$ is the second Stiefel-Whitney class coming from the cohomological invariants of $\mathrm{PGL}_2$.

If $p \neq 2$ for all odd $g$ the cohomological invariants of $\mathscr{H}_g$ are trivial  unless $p$ divides $2g+1$, in which case they are freely generated as a $\operatorname{H}^{\mbox{\tiny{$\bullet$}}}(k_0)$-module by $1$ and a single nonzero invariant of degree $1$.
\end{theorem}
\begin{proof} 
 We begin with the case $p=2$. First, we observe that by proposition (\ref{Gm}) the map $$\left[ P^8 \! \smallsetminus \! \Delta_{1,8} / \mathrm{PGL}_2 \right] \rightarrow \left[ P^8 \! \smallsetminus \! \Delta_{1,8} / \mathrm{PGL}_2 \times \mathrm{G}_m \right]$$ induces an isomorphism on cohomological invariants.

We need to understand whether the $\mathrm{G}_m$-torsor $$\mathscr{H}_3 \rightarrow \left[ P^8 \! \smallsetminus \! \Delta_{1,8} /\mathrm{PGL}_2 \times \mathrm{G}_m \right]$$ generates any new cohomological invariant (note that it cannot kill any existing invariant as it is the composition of a line bundle and an open immersion, both of which induce injective pullbacks).

\vspace{0.5cm}
 Write again $\mathrm{G}= \mathrm{PGL}_2 \times \mathrm{G}_m$. The above amounts to understanding the exact sequence 
 $$ 0 \rightarrow A^{0}_{\mathrm{G}}(P^{8} \! \smallsetminus \! \Delta_{1,8}) \rightarrow A^{0}_{\mathrm{G}}(\mathbb{A}^{9}\! \smallsetminus \! \Delta) \xrightarrow{\partial} A^{0}_{\mathrm{G}}(P^{8} \! \smallsetminus \! \Delta_{1,8}) \xrightarrow{c_1(\mathcal{E})} A^{1}_{\mathrm{G}}(P^{8} \! \smallsetminus \! \Delta_{1,8})$$where $\mathcal{E}$ is the line bundle associated to the $\mathrm{G}_m$ bundle $$\left[ \mathbb{A}^9 \! \smallsetminus \! \Delta /\mathrm{PGL}_2 \times \mathrm{G}_m \right] \rightarrow \left[ P^8 \! \smallsetminus \! \Delta_{1,8} /\mathrm{PGL}_2 \times \mathrm{G}_m \right].$$
 
This is the same as understanding the kernel of the first Chern class of $\mathcal{E}$, and for $p=2$ this is just the first Chern class $t$ of $\mathscr{O}_{P^8}(-1)$  by lemma (\ref{c1Gm}). Then by the formula in lemma (\ref{Gm}) to understand the kernel of $c_1(\mathcal{E})$ we can reduce to $A^1_{\mathrm{PGL}_2}(P^8 \!\smallsetminus\! \Delta_{1,8})$. First we will show that $t, t x_i, t w_2$ each generate a free $\operatorname{H}^{\mbox{\tiny{$\bullet$}}}(k_0)$-module in $A^1_{\mathrm{PGL}_2}(P^8 \!\smallsetminus\! \Delta_{1,8})$, and then we will deal with their $\operatorname{H}^{\mbox{\tiny{$\bullet$}}}(k_0)$-linear combinations. 

Let $\alpha$ be a non-zero element in $\operatorname{H}^{\mbox{\tiny{$\bullet$}}}(k_0)$. The map $$A^{1}_{\mathrm{PGL}_2}(P^8)\rightarrow A^{1}_{\mathrm{PGL}_2}(P^8\! \smallsetminus \! \Delta_{1,8})$$ is injective (its kernel is the image of $A^0_{G}(\Delta_{1,8})$ which is zero), so we know that $ t \alpha$ and $ t  \alpha w_2$ cannot be zero. For the remaining elements we can follow the same reasoning we used in proving the result for $g$ even in \cite[4.1]{Pir2}. For $x_1,x_2, x_3$ we inductively show that they can not be annihilated by $\alpha t$. Consider $$\alpha x_i \in A^0_{\mathrm{PGL}_2}(P^n \! \smallsetminus \! \Delta_{1,n}).$$ We use the new exact sequence $$0 \rightarrow A^{1}_{\mathrm{PGL}_2}(P^n \! \smallsetminus \! \Delta_{2,n}) \rightarrow  A^{1}_{\mathrm{PGL}_2}(P^n \! \smallsetminus \! \Delta_{1,n}) \rightarrow A^{1}_{\mathrm{PGL}_2}(\Delta_{1,n} \! \smallsetminus \! \Delta_{2,n}).$$

 By the compatibility of the boundary map with Chern classes and multiplication with elements coming from $A^{\mbox{\tiny{$\bullet$}}}_{\mathrm{PGL}_2}(k_0)\supset \operatorname{H}^{\mbox{\tiny{$\bullet$}}}(k_0)$, the boundary $\partial(t \alpha x_i)$ restricts to $$t \alpha x_{i-1} \in A^1_{\mathrm{PGL}_2}((P^{n-2}\! \smallsetminus \! \Delta_{1,n-2}) \times P^1)=A^{1}_{\mathrm{PGL}_2}(\Delta_{1,n} \! \smallsetminus \! \Delta_{2,n}).$$ If $\partial(t \alpha x_i)$ is not zero then $t \alpha x_i$ cannot be zero either, and moreover we can restrict to checking that $$t \alpha x_{i-1} \in  A^1_{\mathrm{PGL}_2}(P^{n-2}\! \smallsetminus \! \Delta_{1,n-2})$$ is not zero by lemma \ref{recurs1}.
  Each time we use the reasoning above the degree lowers by one, and eventually we will end up with $$\partial(t \alpha x_1)=t \cdot \alpha \in A^1_{\mathrm{PGL}_2}(P^{n-2i}\! \smallsetminus \! \Delta_{1,n-2i})$$ so it suffices to prove that for $n \geq 2$ the element $t \alpha$ is not zero in the one-codimensional group $A^1_{\mathrm{PGL}_2}((P^{n}\! \smallsetminus \! \Delta_{1,n}) \times P^1)$. This is true as the class of $\Delta_{1,n}$ is equal to zero mod $2$, and thus $A^1_{\mathrm{PGL}_2}((P^{n}\! \smallsetminus \! \Delta_{1,n}) \times P^1)$ has the same elements in degree zero as $A^1_{\mathrm{PGL}_2}(P^{n})$.
  
  Now consider a linear combination $$v=\alpha_0 + \beta w_2 + \alpha_1 x_1 + \alpha_2 x_2 + \alpha_3 x_3.$$
  
 We want to prove that $t v$ is not zero in $A^1_{\mathrm{PGL}_2}(P^8 \!\smallsetminus\! \Delta_{1,8})$. Suppose that $t v=0$. By following the reasoning above, we can take the boundary $\partial$ three times to reduce our element to $t \alpha_3 \in A^1_{\mathrm{PGL}_2}(P^2 \! \smallsetminus \! \Delta_{1,2})$. As above, this element can only be zero if $\alpha_3$ is zero. Now we apply the same idea, taking two boundaries, to get the element $t \alpha_2 \in A^1_{\mathrm{PGL}_2}(P^4 \! \smallsetminus \! \Delta_{1,4})$. Again we conclude that $\alpha_2$ must be zero. Clearly the same reasoning now shows that $\alpha_1$ must be zero too, so we are left with $v=\alpha_0 + \beta w_2$, and the element $t(\alpha_0 + \beta w_2)$ cannot be zero unless $\alpha_0$ and $\beta$ are both zero as the map $$A^{1}_{\mathrm{PGL}_2}(P^8)\rightarrow A^{1}_{\mathrm{PGL}_2}(P^8\! \smallsetminus \! \Delta_{1,8})$$ is injective. 
   This shows that the map $c_1(\mathcal{E})$ is injective when restricted to the free $\operatorname{H}^{\mbox{\tiny{$\bullet$}}}(k_0)$-module generated by $1,x_1, w_2, x_2, x_3$, and if $K$ is the kernel of $c_1(\mathcal{E})$ its projection to the free $\operatorname{H}^{\mbox{\tiny{$\bullet$}}}(k_0)$-module generated by $x_4$ must be injective. Thus we get an exact sequence $$0 \rightarrow A^{0}_{G}(P^{8} \! \smallsetminus \! \Delta_{1,8}) \rightarrow \operatorname{Inv}^{\mbox{\tiny{$\bullet$}}}(\mathscr{H}_3) \rightarrow K \rightarrow 0$$ where $K$ is a submodule of $\operatorname{H}^{\mbox{\tiny{$\bullet$}}}(k_0)\cdot x_4$, shifted up in degree by one as the boundary $\partial$ lowers degree by one. This proves the statement on general fields. 

\vspace{0.5cm}

Let us now assume that $k_0$ is algebraically closed. We want to show that $tx_4$ is equal to $0$ in $A^{1}_{\mathrm{G}}(P^8 \! \smallsetminus \! \Delta_{1,8})$. Then there must be an element $x_5$ in $A^{0}_{\mathrm{G}}(\mathbb{A}^{9}\! \smallsetminus \! \Delta) $ whose boundary $\partial(x_5)$ is equal to $x_4$.

 Note that when $n =2$ the element $\partial (t x_1)$ is indeed zero as $t \in A^{1}_{\mathrm{PGL}_2}(P^2)$ pulls back to zero in $A^1_{\mathrm{PGL}_2}(\Delta_{1,2})$ and there are no elements of degree one in $A^1_{\mathrm{PGL}_2}(P^2)$ when $k_0$ is algebraically closed. This shows that the situation is different than for $x_1, \ldots, x_3$. Even though $k_0$ is now algebraically closed, so that $\operatorname{H}^{\mbox{\tiny{$\bullet$}}}(k_0)=\mathbb{Z}/2\mathbb{Z}$, the matter is a bit more complicated than  in \cite[4.1]{Pir2} for $x_4$ as there are elements of positive degree in $A^0_{\mathrm{PGL}_2}(\Delta_{2,n})$ coming from $A^{\mbox{\tiny{$\bullet$}}}_{\mathrm{PGL}_2}(\operatorname{Spec}(k_0))$. To get around this problem, we make the following consideration. Recall the exact sequence given by the inclusion of $\Delta_{1,8}\! \smallsetminus \! \Delta_{2,8}$ in $P^8\! \smallsetminus \! \Delta_{2,8}$:
 
  $$ 0 \rightarrow A^1_{\mathrm{PGL}_2}(P^8\! \smallsetminus \! \Delta_{2,8}) \rightarrow A^1_{\mathrm{PGL}_2}(P^8 \! \smallsetminus \! \Delta_{1,8} ) \rightarrow A^{1}_{\mathrm{PGL}_2}(\Delta_{1,8}\! \smallsetminus \! \Delta_{2,8}) .$$
 
  There are no elements of degree $4$ in $A^1_{\mathrm{PGL}_2}(P^8\! \smallsetminus \! \Delta_{2,8})$ (because the degree of such elements can be at most the degree of an element of $A^0_{\mathrm{PGL}_2}(\Delta_{2,8})$ plus one, i.e. three), so $t x_4$ is zero if and only if its boundary $\partial(t x_4)$ is zero in $\Delta_{1,8}$. As there are no elements of degree three in $A^0_{\mathrm{PGL}_2}(\Delta_{2,8})$ by lemma (\ref{recurs2}), this is equivalent to asking that $\partial(t x_4)$ is zero in $$A^{1}_{\mathrm{PGL}_2}(\Delta_{1,8}\! \smallsetminus \! \Delta_{2,8})=A^1_{\mathrm{PGL}_2}((P^6 \! \smallsetminus \! \Delta_{1,6})\times P^1).$$
 
  As the boundary of $t x_4$ is the element $t x_3$ in $A^1_{\mathrm{PGL}_2}((P^6 \! \smallsetminus \! \Delta_{1,6})\times P^1)$ we can continue our reasoning on $(P^n \! \smallsetminus \! \Delta_{1,n})\times P^1$. The $P^1$ factor kills all elements of positive degree in $A^{\mbox{\tiny{$\bullet$}}}_{\mathrm{PGL}_2}(P^n\times P^1)$ by proposition (\ref{P^1}) and the projective bundle formula, so we can conclude that $A^0_{\mathrm{PGL}_2}(\Delta_{2,n}\times P^1)$ is trivial using the same argument as in lemma (\ref{recurs2}). This implies that $A^1_{\mathrm{PGL}_2}((P^8\! \smallsetminus \! \Delta_{2,8}) \times P^1)$ can contain elements of degree at most one. Then using   
 $$A^1_{\mathrm{PGL}_2}((P^6\! \smallsetminus \! \Delta_{2,6})\times P^1) \rightarrow A^1_{\mathrm{PGL}_2}((P^6 \! \smallsetminus \! \Delta_{1,6})\times P^1 ) \rightarrow A^{1}_{\mathrm{PGL}_2}((\Delta_{1,6}\! \smallsetminus \! \Delta_{2,6})\times P^1)$$ we conclude that $tx_3$ is zero if and only if its boundary $tx_2$ is zero in  $$A^{1}_{\mathrm{PGL}_2}((\Delta_{1,6}\! \smallsetminus \! \Delta_{2,6})\times P^1)=A^1_{\mathrm{PGL}_2}((P^4\! \smallsetminus \! \Delta_{1,4})\times P^1\times P^1).$$
 
 We can repeat the same reasoning again, reducing our claim to $$t x_1 = 0 \in A^1_{\mathrm{PGL}_2}((P^2\! \smallsetminus \! \Delta_{1,2})\times P^1).$$ 
 As we remarked above when $n=2$ we have $\partial (t x_1)=0 \in A^1_{\mathrm{PGL}_2}(\Delta_{1,2})$, and $A^{1}_{\mathrm{PGL}_2}(P^2 \times P^1)$ only contains elements of degree zero, so looking at the exact sequence $$ A^{1}_{\mathrm{PGL}_2}(P^2 \times P^1) \rightarrow A^1_{\mathrm{PGL}_2}((P^2\! \smallsetminus \! \Delta_{1,2})\times P^1) \xrightarrow{\partial} A^1_{\mathrm{PGL}_2}(\Delta_{1,2})$$ we conclude that $t x_1$ must be equal to $0$.
 
 \vspace{0.5cm}
 
Finally we deal with the case $p \neq 2$. Denote by $\mathcal{E}_n$ the line bundle obtained by extending the $\mathrm{G}_m$ bundle $$\left[ (\mathbb{A}^{n+1}\! \smallsetminus \! \Delta) /\mathrm{PGL}_2 \times \mathrm{G}_m \right]\rightarrow \left[(P^n\! \smallsetminus \! \Delta_{1,n})/\mathrm{PGL}_2 \times \mathrm{G}_m \right]$$ using again the exact sequence above we only have (at worst) to check whether the products $c_1(\mathcal{E}_n)\cdot 1, c_1(\mathcal{E}_n)\cdot x_1$ are $\operatorname{H}^{\mbox{\tiny{$\bullet$}}}$-linearly dependent inside $A^{1}_{\mathrm{PGL}_2\times \mathrm{G}_m}(P^n \! \smallsetminus \! \Delta_{1,n})$, in which case we would see some new cohomological invariant appearing.

 Consider a linear combination $v=\alpha + \beta x_1$, and assume that $c_1(\mathcal{E}_n)\cdot v=0$. We can take the boundary of $c_1(\mathcal{E}_n)\cdot v$, which by (\ref{c1Gm}) is equal to $$(t-2s)\cdot \beta \in A^{1}_{\mathrm{PGL}_2\times \mathrm{G}_m}((P^{n-2} \! \smallsetminus \! \Delta_{1,n-2})\times P^1).$$
 
  For this element to be zero it would have to be equal to a multiple of the class of $\Delta_{1,n}$ in $A^{1}_{\mathrm{PGL}_2\times \mathrm{G}_m}(P^n \! \smallsetminus \! \Delta_{1,n})$, which never happens as this class is a multiple of $t$ and $2s$ is not divisible by $p$. This shows that $\beta=0$. The we are left with $v=\alpha$ for some $\alpha \in \operatorname{H}^{\mbox{\tiny{$\bullet$}}}$, and again $$(t-2s)\cdot \alpha \in A^{1}_{\mathrm{PGL}_2\times \mathrm{G}_m}(P^{n} \! \smallsetminus \! \Delta_{1,n})$$ cannot be zero for the same reason.
\end{proof}

\end{document}